\documentclass[reqno,11pt,oneside]{amsart}
\DeclareFontEncoding{OT2}{}{}
\DeclareFontSubstitution{OT2}{cmr}{m}{n}
\DeclareFontFamily{OT2}{cmr}{\hyphenchar\font45 }
\DeclareFontShape{OT2}{cmr}{m}{n}{
<5><6><7><8><9>gen*wncyr
<10><10.95><12><14.4><17.28><20.74><24.88>wncyr10}{}
\DeclareFontShape{OT2}{cmr}{b}{n}{
<5><6><7><8><9>gen*wncyb
<10><10.95><12><14.4><17.28><20.74><24.88>wncyb10}{}
\DeclareMathAlphabet{\mathcyr}{OT2}{cmr}{m}{n}
\DeclareMathAlphabet{\mathcyb}{OT2}{cmr}{b}{n}
\SetMathAlphabet{\mathcyr}{bold}{OT2}{cmr}{b}{n}
\usepackage{fancybox,ascmac,comment,xcolor}
\usepackage{fullpage}
\usepackage{amssymb,mathrsfs}
\usepackage{amsfonts}
\usepackage{amsmath}
\usepackage{amsthm}
\usepackage{mathtools}
\usepackage{enumerate}

\usepackage{mleftright}
\mleftright
\usepackage[all]{xy}
\usepackage{tikz}
\usepackage{tikz-cd}
\mathtoolsset{showonlyrefs=true}
\numberwithin{equation}{section}
\makeatletter
\newcommand{\shortmathcal}[1]{\@tfor\ch:=#1\do{
\expandafter\edef\csname c\ch\endcsname{\noexpand\mathcal{\ch}}
}}
\newcommand{\shortmathbb}[1]{\@tfor\ch:=#1\do{
\expandafter\edef\csname bb\ch\endcsname{\noexpand\mathbb{\ch}}
}}
\newcommand{\shortmathbf}[1]{\@tfor\ch:=#1\do{
\expandafter\edef\csname b\ch\endcsname{\noexpand\mathbf{\ch}}
}}
\newcommand{\shortboldsymbol}[1]{\@tfor\ch:=#1\do{
\expandafter\edef\csname bs\ch\endcsname{\noexpand\boldsymbol{\ch}}
}}
\newcommand{\shortmathfrak}[1]{\@tfor\ch:=#1\do{
\expandafter\edef\csname f\ch\endcsname{\noexpand\mathfrak{\ch}}}}
\newcommand{\shortmathscr}[1]{\@tfor\ch:=#1\do{
\expandafter\edef\csname s\ch\endcsname{\noexpand\mathscr{\ch}}}}
\newcommand{\shortmathrm}[1]{\@tfor\ch:=#1\do{
\expandafter\edef\csname r\ch\endcsname{\noexpand\mathrm{\ch}}
}}
\let\@@span\span
\def\sp@n{\@@span\omit\advance\@multicnt\m@ne}
\makeatother
\shortmathcal{ABCDEFGHIJKLMNOPQRSTUVWXYZ}
\shortmathbb{ABCDEFGHIJKLMNOPQRSTUVWXYZ}
\shortmathbf{ABCDEFGHIJKLMNOPQRSTUVWXYZabcdefghijklmnopqrstuvwxyz}
\shortboldsymbol{ABCDEFGHIJKLMNOPQRSTUVWXYZabcdefghijklmnopqrstuvwxyz}
\shortmathfrak{ABCDEFGHIJKLMNOPQRSTUVWXYZabcdefghjklmnopqrstuvwxyz}
\shortmathscr{ABCDEFGHIJKLMNOPQRSTUVWXYZ}
\shortmathrm{ABCDEFGHIJKLMNOPQRSTUVWXYZabcdefghijklmnopqrstuvwxyz}

\newcommand{\id}{\mathrm{id}}

\newcommand{\emp}{\varnothing}
\newcommand{\ep}{\varepsilon}

\DeclareMathOperator{\Fil}{Fil}

\newcommand{\sh}{\mathbin{\mathcyr{sh}}}

\newcommand{\MU}{\mathsf{MU}}
\newcommand{\LU}{\mathsf{LU}}
\newcommand{\BIMU}{\mathsf{BIMU}}
\newcommand{\as}{\mathsf{as}}
\newcommand{\al}{\mathsf{al}}
\newcommand{\mmu}{\mathsf{mu}}

\newcommand{\invmu}{\mathsf{invmu}}

\newcommand{\adari}{\mathsf{adari}}
\newcommand{\preari}{\mathsf{preari}}
\newcommand{\expari}{\mathsf{expari}}
\newcommand{\logari}{\mathsf{logari}}

\newcommand{\invgari}{\mathsf{invgari}}
\newcommand{\fragari}{\mathsf{fragari}}
\newcommand{\fragira}{\mathsf{fragira}}
\newcommand{\invgami}{\mathsf{invgami}}
\newcommand{\invgani}{\mathsf{invgani}}
\newcommand{\gira}{\mathsf{gira}}
\newcommand{\girat}{\mathsf{girat}}

\newcommand{\garit}{\mathsf{garit}}
\newcommand{\gari}{\mathsf{gari}}
\newcommand{\arit}{\mathsf{arit}}
\newcommand{\ari}{\mathsf{ari}}
\newcommand{\irat}{\mathsf{irat}}
\newcommand{\gaxit}{\mathsf{gaxit}}
\newcommand{\gaxi}{\mathsf{gaxi}}
\newcommand{\axit}{\mathsf{axit}}

\newcommand{\gamit}{\mathsf{gamit}}

\newcommand{\amit}{\mathsf{amit}}

\newcommand{\ganit}{\mathsf{ganit}}
\newcommand{\gani}{\mathsf{gani}}
\newcommand{\anit}{\mathsf{anit}}

\newcommand{\GARI}{\mathsf{GARI}}
\newcommand{\ARI}{\mathsf{ARI}}

\newcommand{\fes}{\mathfrak{es}}
\newcommand{\fos}{\mathfrak{os}}
\newcommand{\fess}{\mathfrak{ess}}
\newcommand{\foss}{\mathfrak{oss}}
\newcommand{\dess}{\ddot{\mathfrak{e}}\mathfrak{ss}}
\newcommand{\doss}{\ddot{\mathfrak{o}}\mathfrak{ss}}

\newcommand{\foz}{\mathfrak{oz}}

\newcommand{\ful}[1]{{}_{#1}\lceil}
\newcommand{\fll}[1]{{}_{#1}\lfloor}
\newcommand{\fur}[1]{\rceil_{#1}}
\newcommand{\flr}[1]{\rfloor_{#1}}

\renewcommand{\neg}{\mathsf{neg}}
\newcommand{\swap}{\mathsf{swap}}
\newcommand{\anti}{\mathsf{anti}}
\newcommand{\pari}{\mathsf{pari}}
\newcommand{\push}{\mathsf{push}}

\newcommand{\mantar}{\mathsf{mantar}}
\newcommand{\gantar}{\mathsf{gantar}}

\newcommand{\crash}{\mathsf{crash}}
\renewcommand{\slash}{\mathsf{slash}}
\newcommand{\leng}{\mathsf{leng}}

\newcommand{\dTo}{\mathfrak{T}\ddot{\mathfrak{o}}}
\newcommand{\der}{\mathsf{der}}

\newcommand{\dro}{\fr\ddot{\fo}}
\newcommand{\Rush}{\mathsf{Rush}}
\newcommand{\ORush}{\fO\text{-}\Rush}
\newcommand{\Esena}{\fE\text{-}\mathsf{sena}}
\newcommand{\baral}{\underline{\al}}
\newcommand{\barfol}{\underline{\mathfrak{ol}}}

\newcommand{\Emantar}{\fE\text{-}\mantar}
\newcommand{\Omantar}{\fO\text{-}\mantar}
\newcommand{\Enegpush}{\fE\text{-}\neg\push}
\newcommand{\Eneg}{\fE\text{-}\neg}
\newcommand{\Eter}{\fE\text{-}\mathsf{ter}}
\newcommand{\Epush}{\fE\text{-}\push}
\newcommand{\Eswap}{\fE\text{-}\swap}
\newcommand{\dmr}{\mathfrak{dmr}}
\newcommand{\krv}{\mathfrak{krv}}
\newcommand{\swamu}{\mathsf{swamu}}
\newcommand{\answamu}{\mathsf{answamu}}

\newcommand{\fol}{\mathfrak{ol}}
\newcommand{\lead}{\mathsf{lead}}
\newcommand{\preira}{\mathsf{preira}}

\theoremstyle{definition}
\newtheorem{theorem}{Theorem}[section]
\newtheorem{proposition}[theorem]{Proposition}
\newtheorem{lemma}[theorem]{Lemma}
\newtheorem{corollary}[theorem]{Corollary}

\newtheorem{definition}[theorem]{Definition}
\theoremstyle{remark}
\newtheorem{remark}[theorem]{Remark}

\allowdisplaybreaks
\title{Ecalle's senary relation and dimorphic structures}
\author{Hanamichi Kawamura}
\address[Hanamichi Kawamura]{Department of Mathematics, Graduate School of Science, Tokyo University of Science, 1-3 Kagurazaka, Shinjuku-ku, Tokyo, 162-8601, Japan}
\email{1125512@ed.tus.ac.jp}
\subjclass[2020]{17B40.}
\keywords{bimoulds, senary relation, dimorphy, flexion}
\allowdisplaybreaks
\begin{document}
\begin{abstract}
  This paper gives a completed proof of \'{E}calle's \emph{senary relation} in the original form. As an application, we obtain another proof of the fact that the double shuffle Lie algebra injects into the Kashiwara--Vergne Lie algebra.
  We also give two explicit isomorphisms between some dimorphic Lie algebras and their linearized version.
\end{abstract}
\maketitle
\section{Introduction}
This is the sequel of the author's work \cite{kawamura25}, as a part of the project describing the theory of flexion units related to number theory.
In this article, we particularly focus on the \emph{senary relation}, which firstly appeared in \'{E}calle \cite[(3.58)]{ecalle11}.
The original statement is in the context of \emph{dimorphic structures} in his words, describing the strucures of objects simulatenously obeying two types of symmetries.
Among many symmetries \'{E}calle considered, we investigate the \emph{twisted symmetries}, defined once we fix each flexion unit. Together with the usual alternality, they form a dimorphy, and our first main theorem stated below says that it implies the senary relation. For the symbols appearing in below statements, see the following sections and \cite{kawamura25}.

\begin{theorem}[$=$ Proposition \ref{prop:enegpush} $+$ Theorem \ref{thm:357}]\label{thm:main}
  Let $\fE$ be a flexion unit, $\fO$ its conjugate and $B\in\ARI_{\baral/\barfol}$.
  Then $B$ satisfies the senary relation
  \begin{equation}\label{eq:senary}
    \Eter(B)=(\push\circ\mantar\circ\Eter\circ\mantar)(B).
  \end{equation}
\end{theorem}

Moreover, in the same framework, we show that there is a constructible bijection between the space of what enjoys the senary relation and much simpler one.

\begin{theorem}[$=$ Corollary \ref{cor:push_senary}]\label{thm:push_senary}
  For any flexion unit $\fE$, both the maps $\adari(\fess)$ and $\adari(\dess)$ gives an isomorphism $\ARI_{\push}\to\ARI_{\Esena}$ of $R$-modules.
\end{theorem}

\begin{corollary}
  For any flexion unit $\fE$, the module $\ARI_{\Esena}$ is a Lie subalgebra of $\ARI$.
\end{corollary}

As a key ingredient of the proof of these claims, we also show the bisymmetrality of the secondary bimould, in the appendix.

\begin{theorem}[$=$ Theorem \ref{thm:bisymmetral2}; {\cite[p.~80]{ecalle11}}]\label{thm:bisymmetral}
  The bimould $\fess$ is bisymmetral; namely, both $\fess$ and its swappee $\doss$ are symmetral.
\end{theorem}

By specializing the flexion unit, these results have an application which connects double shuffle structure in number theory and Kashiwara--Vergne structure in Lie theory.
We exhibit an important one of recent results.

\begin{theorem}[{\cite[Theorem 1]{schneps25}, \cite[Theorem 0.19]{ef25}}]\label{thm:dmr_krv}
  There is an injection $\dmr_{0}\hookrightarrow\krv_{2}$ of Lie algebras.
\end{theorem}
More precisely, the above injection is already shown in \cite[Theorem 1.1]{schneps12} assuming a special case of the senary relation \eqref{eq:senary}.
Since our result (Theorem \ref{thm:main}) includes that assumption, it also gives a new proof of Theorem \ref{thm:dmr_krv}.

\section*{Acknowledgments}
The author is grateful to Professor Henrik Bachmann, Professor Hidekazu Furusho and Professor Minoru Hirose for giving him opportunity to give a talk in Kagoshima mould theory seminar and its preparation seminar. This work is done in the preparation for those talks.
The author also would like to thank Atsushi Matsuura and his ChatGPT for giving him a crucial idea to prove Lemma \ref{lem:negelon}.

\section{Preliminaries}
Throughout this paper, we use the notations in \cite{kawamura25} without declaring.
We fix a deck $(\bbK,\cX)$ for bimoulds, and a $\bbK$-algebra $R$, so that we can omit $R$ from several notations. A flexion unit $\fE$ is also fixed unless mentioned, and denote by $\fO$ its conjugate unit.\\

\subsection{Adjoint action $\adari$}
First we define the adjoint action of $\GARI$ on $\ARI$:
\begin{definition}
  For each $M\in\GARI$, we define $\adari(M)\colon\ARI\to\ARI$ by
  \[A\mapsto A+\sum_{n=1}^{\infty}\frac{1}{n!}\underbrace{\ari(\logari(M),\cdots,\ari(\logari(M)}_{n},A)\cdots).\]
\end{definition}
This is well-defined since the $\ari$-bracket preserves the length filtration in the sense that, for any $A\in\Fil^{\ge m}\BIMU$ and $B\in\Fil^{\ge n}\BIMU$, it holds that $\ari(A,B)\in\Fil^{\ge m+n}\BIMU$.\\

From the general theory of adjoint actions, we see that
\begin{align}
  \adari(\gari(M,N))(A)&=(\adari(M)\circ\adari(N))(A),\\
  \adari(M)(\ari(A,B))&=\ari(\adari(M)(A),\adari(M)(B))
\end{align}
and
\[\adari(M)(A)=\gari(\preari(M,A),\invgari(M))\]
hold for every $M,N\in\GARI$ and $A,B\in\ARI$.
Note that, the second expression above relies on the fact that $\preari$ gives the linearization of $\gari$ \cite[Remark 3.16]{kawamura25}.

\subsection{Twisted symmetries}
Recall that a bimould $A\in\ARI$ is said to be \emph{alternal} if $A(\ba\sh\bb)=0$ holds for arbitrary non-empty sequences $\ba$ and $\bb$.
\begin{definition}[{\cite[\S 3.4]{ecalle11}}]
  Let $A\in\ARI$. 
  We say that $A$ is \emph{$\fO$-alternal} if $\ganit(\foz)^{-1}(A)$ is alternal.
\end{definition}

\begin{remark}
  Komiyama suggested, in his personal note, an alternative definition of the $\fO$-alternality via introducing a recursive product similar to the quasi-shuffle product.
\end{remark}

\begin{lemma}\label{lem:garit_os}
  We have
  \[\ganit(\fos)\circ\gamit(\fos)^{-1}=\garit(\invmu(\fos)).\]
  Especially, this operator preserves symmetrality.
\end{lemma}
\begin{proof}
  From the explicit formula
  \[\fos\binom{u_{1},\ldots,u_{r}}{v_{1},\ldots,v_{r}}=\fO\binom{u_{1}}{v_{1}}\fO\binom{u_{1}+u_{2}}{v_{2}-v_{1}}\cdots\fO\binom{u_{1}+\cdots+u_{r}}{v_{r}-v_{r-1}},\]
  we have
  \begin{align}
    \ganit(\fos)(\fO)(w_{1},\ldots,w_{r})
    &=\fO(w_{1}\fur{w_{2}\cdots w_{r}})\fos(\fll{w_{1}}(w_{2}\cdots w_{r}))\\
    &=\fos(w_{1},\ldots,w_{r})
  \end{align}
  unless $r=0$. Hence we have
  \begin{equation}\label{eq:ganit_os}
  \ganit(\fos)(\fO)=\fos-1
  \end{equation}
  On the other hand, we know $\invgami(\fos)=\pari(\foz)$ because $\anti\circ\invgami=\invgani\circ\anti$ and $\invgani(\foz)=(\pari\circ\anti)(\fos)$ (\cite[Proposition 5.3 (1)]{kawamura25}), while the former follows from the identity $\gamit(\anti(X))\circ\anti=\anti\circ\ganit(X)$.
  Then, using \cite[Proposition 3.10]{kawamura25},
  we obtain
  \begin{align}
    &\ganit(\fos)\circ\gamit(\fos)^{-1}\\
    &=\ganit(\fos)\circ\gamit(\pari(\foz))\\
    &=\gaxit((1,\fos))\circ\gaxit((\pari(\foz),1))\\
    &=\gaxit(\gaxi((\pari(\foz),1),(1,\fos)))\\
    &=\gaxit(\ganit(\fos)(\pari(\foz)),\fos).
  \end{align}
  Hence it suffices to show $\ganit(\fos)(\pari(\foz))=\invmu(\fos)$.
  This equality is shown as
  \begin{align}
    \ganit(\fos)(\pari(\foz))
    &=\ganit(\fos)(\invmu(1+\fO))\\
    &=(\invmu\circ\ganit(\fos))(1+\fO)\\
    &=\invmu(1+\ganit(\fos)(\fO))\\
    &=\invmu(\fos),
  \end{align}
  where the final equality is due to \eqref{eq:ganit_os}.
  To show the rest assertion, we take an arbitrary $S\in\GARI_{\as}$.
  Then we have
  \[\garit(\invmu(\fos))(S)=\mmu(\gari(S,\invmu(\fos)),\fos),\]
  from the definition of $\gari$.
  Since $\fos$ is symmetral (\cite[Proposition 5.3 (2)]{kawamura25}) and both $\gari$ and $\mmu$ preserves symmetrality, we obtain the result.
\end{proof}

The following fact is mentioned in the 43rd footnote of \cite{ecalle11} without a proof.
\begin{lemma}
  A bimould $A\in\ARI$ is $\fO$-alternal if and only if $\gamit(\foz)^{-1}(A)$ is alternal.
\end{lemma}
\begin{proof}
  For any bimould $B$, one deduces
  \[\ganit(\foz)^{-1}(A) = (\gamit(\fos)\circ\ganit(\fos)^{-1}\circ\gamit(\foz)^{-1})(A)\]
  from the explicit formula \cite[Corollary 3.11]{kawamura25} stating
  \[\gaxit(X,Y)=\gamit(X)\circ\ganit(\gamit(X)^{-1}(Y))=\ganit(Y)\circ\gamit(\ganit(Y)^{-1}(X)).\]
  Hence it is sufficient to show that the operator $\gamit(\fos)\circ\ganit(\fos)^{-1}$ preserves alternality.
  Indeed it is $\garit(\invmu(\fos))^{-1}$ and preserves symmetrality, according to Lemma \ref{lem:garit_os}.
  Hence we conclude that $\garit(\invmu(\fos))^{-1}\colon\MU_{\as}\to\MU_{\as}$ induces its derivative $\LU_{\al}\to\LU_{\al}$, which is $\garit(\invmu(\fos))^{-1}$ itself as $\garit(\bullet)$ is always linear.
\end{proof}

\section{Dimorphic transportation}
In this section, we describe \'{E}calle's dimorphic transportation via $\adari(\fess)$ or $\adari(\dess)$.
Our argument proceeds similarly to Schneps' one \cite{schneps15}, but it works for any flexion unit, not only polar ones.
The following formula is an alternative form of \'{E}calle's \emph{first fundamental identity}, which is shown by Schneps.

\begin{theorem}[{\cite[Corollary 2.8.6]{schneps15}}]\label{thm:first}
    For $A,B\in\GARI$, we have
    \begin{equation}\label{eq:first_fundamental}
      \swap(\fragari(\swap(A),\swap(B)))=\ganit(\crash(B))(\fragari(A,B)).
    \end{equation}
\end{theorem}

Using this formula, we derive the following \emph{second fundamental identity} in the general setting. Denote by $\ARI_{\push}$ the set of all $\push$-invariants in $\ARI$.

\begin{corollary}[{\cite[(5.17)]{ecalle11}}]\label{cor:second}
  For $A\in\ARI_{\push}$, we have
  \[(\swap\circ\adari(\fess))(A)=(\ganit(\foz)\circ\adari(\dess)\circ\swap)(A)\]
  and
  \[(\swap\circ\adari(\dess))(A)=(\ganit(\foz)\circ\adari(\foss)\circ\swap)(A).\]
\end{corollary}
\begin{proof}
  Since
  \[(\swap\circ\adari(\fess))(A)=\swap(\fragari(\preari(\fess,A),\fess)),\]
  we apply \eqref{eq:first_fundamental} as
  \begin{align}
    (\swap\circ\adari(\fess))(A)
    &=\ganit(\crash(\swap(\fess)))(\fragari(\swap(\preari(\fess,A)),\swap(\fess)))\\
    &=\ganit(\crash(\doss))(\fragari(\swap(\preari(\fess,A)),\dess))\\
    &=\ganit(\foz)(\fragari(\swap(\preari(\fess,A)),\dess)),
  \end{align}
  where we used \cite[Theorem 1.2]{kawamura25}.
  Thus it suffices to show $\swap(\preari(\fess,A))=\preari(\dess,\swap(A))$, and it is done in a quite same way as \cite[(4.5.5)]{schneps15}.
  The second equality in the statement is shown in the same way, as \cite[Theorem 1.2]{kawamura25} implies that $\crash(\foss)=\crash(\doss)=\foz$.
\end{proof}

We define the following two type of spaces:
\[
  \ARI_{\baral/\baral}\coloneqq\left\{A\in\ARI_{\al}~\left|~\begin{array}{c}A\binom{u_{1}}{v_{1}}=A\binom{-u_{1}}{-v_{1}},\\ \swap(A)\in\ARI_{\al}\end{array}\right.\right\},\]
and\[  \ARI_{\baral/\barfol}\coloneqq\left\{A\in\ARI_{\al}~\left|~\begin{array}{c}A\binom{u_{1}}{v_{1}}=A\binom{-u_{1}}{-v_{1}},\\ \swap(A)\in\ARI_{\fol}\end{array}\right.\right\}.
\]
where the space $\ARI_{\fol}$ denotes the set of all $\fO$-alternals.\\

We prepare two lemmas:
\begin{lemma}\label{lem:negpush}
  Every element of $\ARI_{\baral/\baral}$ is $\neg$- and $\push$-invariant.
\end{lemma}
\begin{proof}
  There is a completed proof in \cite[Lemma 2.5.5]{schneps15}.
\end{proof}

\begin{lemma}\label{lem:lead}
    For $A\in\BIMU$, we put
    \[\lead(A)\coloneqq\leng_{m}(A),\]
    where $m$ is the minimum value of $\{r\mid \leng_{r}(A)\neq 0\}$.
    Then we have $\lead(A)\in\ARI_{\baral/\baral}$ for every $A\in\ARI_{\baral/\barfol}$.
\end{lemma}
\begin{proof}
    The condition $A\in\ARI_{\baral}$ immediately implies $\lead(A)\in\ARI_{\baral}$.
    Let us investigate the $\fO$-alternality at length $m$: focus on the alternal bimould $(\ganit(\foz)^{-1}\circ\swap)(A)$ is alternal. Since it has the form
    \[(\ganit(\foz)^{-1}\circ\swap)(A)=\swap(\lead(A))+(\text{length}>m\text{ terms})\]
    and the alternality condition is homogeneous, we obtain $\swap(\lead(A))\in\ARI_{\al}$.
\end{proof}

\begin{theorem}[{\cite[p.~95]{ecalle11}}]\label{thm:dimorphic_transport}
  Each of the maps $\adari(\fess)$ and $\adari(\dess)$ gives an $R$-linear bijection $\ARI_{\baral/\baral}\to\ARI_{\baral/\barfol}$.
\end{theorem}
\begin{proof}
  As the $R$-linearity is clear from the construction of $\adari$, we prove the required condition for the image.
  Let $A$ belong to $\ARI_{\baral/\baral}$. The parity condition for $\adari(\fess)(A)$ is easy as $\leng_{1}\circ\adari(\fess)=\leng_{1}$.
  Since the $\gari$-product preserves symmetrality (\cite[Proposition 3.17]{kawamura25}) and $\expari(\ARI_{\al})=\GARI_{\as}$ (\cite[Theorem A.7]{komiyama21}), the bimould
  \[\adari(\fess)(A)=\logari(\gari(\fess,\expari(A),\invgari(\fess)))\]
  is alternal as $\fess\in\GARI_{\as}$ (Theorem \ref{thm:bisymmetral}).
  Next, using the $\push$-invariance of $A$, which is obtained by the assumption and Lemma \ref{lem:negpush}, we can apply Corollary \ref{cor:second}, which immediately yields $\swap(A)\in\ARI_{\fol}$. 

    Finally let us show that the inverse image $\adari(\fess)^{-1}(\ARI_{\baral/\barfol})$ is contained in $\ARI_{\baral/\baral}$.
    Take an arbitrary $B\in\ARI_{\baral/\barfol}$.
    We define a sequence $(A_{0},A_{1},\ldots)$ of bimoulds by the recursive rule, $A_{0}\coloneqq B$ and
    \[A_{n+1}\coloneqq A_{n}-(\adari(\fess)\circ\lead)(A_{n}).\]
    We can show that each $A_{n}$ belongs to $\ARI_{\baral/\barfol}$ by induction and using Lemma \ref{lem:lead}.
    Then it is easily shown that
    \begin{align}
        B
        &=A_{0}\\
        &=\sum_{n=0}^{\infty}(A_{n}-A_{n+1})\\
        &=\sum_{n=0}^{\infty}(\adari(\fess)\circ\lead)(A_{n})\\
        &=\adari(\fess)\left(\sum_{n=0}^{\infty}\lead(A_{n})\right).
    \end{align}
    Here, since for each $\lead(A_{n})$ we have the unique strictly increasing sequence $\{\ell_{n}\}_{n\ge 0}$ of positive integers such that $\lead(A_{n})\in\BIMU_{\ell_{n}}$ and $\lead(\adari(\fess)(\lead(A_{n})))$ also belongs to $\BIMU_{\ell_{n}}$,     
    the above infinite sum converges with respect to the length filtration of $\BIMU$.
    From the above expression, we obtain $B\in\ARI_{\baral/\baral}$, because each component $\lead(A_{n})$ belongs to $\ARI_{\baral/\baral}$ by Lemma \ref{lem:lead} and the previous induction.
    We can use the same argument to show the $\dess$ case.
\end{proof}

\begin{lemma}\label{lem:alal_lie}
  The $R$-module $\ARI_{\baral/\baral}$ becomes a Lie subalgebra of $\ARI$.
\end{lemma}
\begin{proof}
  It was shown by Schneps \cite[Theorem 2.5.6]{schneps15}.
\end{proof}

Since Theorem \ref{thm:dimorphic_transport} and Lemma \ref{lem:alal_lie} holds, we obtain the following.

\begin{corollary}
  The $R$-module $\ARI_{\baral/\barfol}$ is a Lie subalgebra of $\ARI$.
\end{corollary}

\section{Subsymmetries following dimorphy}
Here we prove the counterpart of Lemma \ref{lem:negpush} for the dimorphic algebra $\ARI_{\baral/\barfol}$.
We prepare the following four operators:
\begin{align}
    \Omantar&\coloneqq\ganit(\foz)\circ\mantar\circ\ganit(\foz)^{-1},\\
    \Enegpush&\coloneqq\mantar\circ\swap\circ(\Omantar)\circ\swap,\\
    \Eneg&\coloneqq\neg\circ\adari(\fes),\\
    \Epush&\coloneqq(\Eneg)\circ(\Enegpush).
\end{align}
Be careful that the third factor of $\Enegpush$ is $\Omantar$, not\footnote{The original definition \cite[(3.51)]{ecalle11} looks like there is a mistake, because in general $\fO$-alternal moulds are not $\Emantar$-invariant.} $\Emantar$.

\begin{proposition}[{\cite[p.~66]{ecalle11}}]\label{prop:enegpush}
    Every element $A\in\ARI_{\baral/\barfol}$ is $\Enegpush$-invariant.
\end{proposition}

\begin{proof}
    By definition and $\mantar$-invariance for alternals (\cite[Proposition 2.14]{kawamura25}), we easily check that every $\fO$-alternal bimould is $\Omantar$-invariant.
    Hence we have $(\Omantar\circ\swap)(A)=\swap(A)$ and
    \begin{align}
        \Enegpush(A)
        &=(\mantar\circ\swap\circ\Omantar\circ\swap)(A)\\
        &=\mantar(A)\\
        &=A.
    \end{align}
\end{proof}

\begin{theorem}
    Let $A$ be an element of $\ARI_{\baral/\barfol}$.
    Then $A$ is $\Epush$-invariant. 
\end{theorem}

\begin{proof}
    By Proposition \ref{prop:enegpush}, it suffices to show the $\Eneg$-invariance. By the identity $\slash(\fess)=\fes$ (\cite[Theorem 1.3]{kawamura25}), we obtain $\adari(\fes)=\adari(\neg(\fess))\circ\adari(\fess)^{-1}$. Moreover, since $\neg$ is distributive for $\preari$ and $\gari$, for any bimould $X\in\ARI$ and $Y\in\GARI$ we have
    \begin{align}
        \adari(\neg(Y))(\neg(X))
        &=\fragari(\preari(\neg(Y),\neg(X)),\neg(Y))\\
        &=\fragari(\neg(\preari(Y,X)),\neg(Y))\\
        &=\gari(\neg(\preari(Y,X)),(\invgari\circ\neg)(Y))\\
        &=\gari(\neg(\preari(Y,X)),(\neg\circ\invgari)(Y))\\
        &=\neg(\gari(\preari(Y,X),\invgari(Y)))\\
        &=\neg(\fragari(\preari(Y,X),Y))\\
        &=\neg(\adari(Y)(X)).
    \end{align}
    Therefore the operator $\Eneg$ is equal to
    \[\Eneg=\neg\circ\adari(\fes)=\neg\circ\adari(\neg(\fess))\circ\adari(\fess)^{-1}=\adari(\fess)\circ\neg\circ\adari(\fess)^{-1}.\]
    By the assumption and Theorem \ref{thm:dimorphic_transport}, the operator $\adari(\fess)^{-1}(A)$ belongs to $\ARI_{\baral/\baral}$ and thus $\neg$-invariant from Lemma \ref{lem:negpush}.
    This shows the $\Eneg$-invariance of $A$.
\end{proof}

\section{Senary relation}
In this section, we give a proof of \'{E}calle's senary relation 
\[\Eter(B)=(\push\circ\mantar\circ\Eter\circ\mantar)(B)\]
for a $\Epush$-invairant bimould $B$.
Here, the operator $\Eter$ is defined by $\Eter(B)(\emp)\coloneqq B(\emp)$ and
\begin{multline}\Eter(B)(w_{1},\ldots,w_{r})\\\coloneqq B(w_{1},\ldots,w_{r})-B(w_{1},\ldots,w_{r-1})\fE(w_{r})+B((w_{1},\ldots,w_{r-1})\fur{w_{r}})\fE(\fll{w_{1},\ldots,w_{r-1}}w_{r})\end{multline}
for a bimould $B$.\\

In the following arguments, we often use the $\swap$-conjugation and $\swap\circ\anti$-conjugation of the product $\mmu$:
\begin{align}
    \swamu(A,B)&\coloneqq\swap(\mmu(\swap(A),\swap(B))),\\
    \answamu(A,B)&\coloneqq(\anti\circ\swap)(\mmu((\swap\circ\anti)(A),(\swap\circ\anti)(B)))\\
    &=\anti(\swamu(\anti(A),\anti(B))).
\end{align}
The next proposition is easy to prove but saying these conjugate products are very useful.

\begin{proposition}\label{prop:swamu_answamu}
    Let $A$ and $B$ be bimoulds.
    Then we have
    \begin{equation}\label{eq:swamu}
      \swamu(A,B)(\bw)=\sum_{\bw=\ba\bb}A(\ful{\ba}\bb)B(\ba\flr{\bb})
    \end{equation}
    and
    \begin{equation}\label{eq:answamu}
      \answamu(A,B)(\bw)=\sum_{\bw=\ba\bb}A(\ba\fur{\bb})B(\fll{\ba}\bb).
    \end{equation}
    Moreover, for bimoulds $A,B$ and $C$, the following are true:
    \begin{enumerate}
        \item $\swamu(\mmu(A,B),C)=\mmu(\swamu(A,C),B)$ if $A\in\ARI$.
        \item $\answamu(\mmu(A,B),C)=\mmu(A,\answamu(B,C))$ if $B\in\ARI$.
        \item $\swamu(\answamu(A,B),C)=\answamu(\swamu(A,C),B)$ if $A\in\ARI$.
    \end{enumerate}
\end{proposition}
\begin{proof}
    Write $\bw=(w_{1},\ldots,w_{r})$.
    First, we show \eqref{eq:swamu} by direct computation:
    \begin{align}
      &\swamu(A,B)(\bw)\\
      &=\mmu(\swap(A),\swap(B))\binom{v_{r},v_{r-1}-v_{r},\ldots,v_{1}-v_{2}}{u_{1}+\cdots+u_{r},u_{1}+\cdots+u_{r-1},\ldots,u_{1}}\\
      &=\begin{multlined}[t]\sum_{i=0}^{r}\swap(A)\binom{v_{r},v_{r-1}-v_{r},\ldots,v_{r-i+1}-v_{r-i+2}}{u_{1}+\cdots+u_{r},u_{1}+\cdots+u_{r-1},\ldots,u_{1}+\cdots+u_{r-i+1}}\\\cdot\swap(B)\binom{v_{r-i}-v_{r-i+1},\ldots,v_{1}-v_{2}}{u_{1}+\cdots+u_{r-i},\ldots,u_{1}}\end{multlined}\\
      &=\sum_{i=0}^{r}A\binom{u_{1}+\cdots+u_{r-i+1},u_{r-i+2},\ldots,u_{r}}{v_{r-i+1},v_{r-i+2},\ldots,v_{r}}B\binom{u_{1},u_{2},\ldots,u_{r-i}}{v_{1}-v_{r-i+1},v_{2}-v_{r-i+1},\ldots,v_{r-i}-v_{r-i+1}}\\
      &=\sum_{\bw=\ba\bb}A(\ful{\ba}\bb)B(\ba\flr{\bb}).
    \end{align}
    Next, writing 
    \[\overleftarrow{w_{1}\cdots w_{r}}\coloneqq w_{r}\cdots w_{1}\]
    we show the second identity \eqref{eq:answamu} by using the first one as
    \begin{align}
      \answamu(A,B)(\bw)
      &=\swamu(\anti(A),\anti(B))(\overleftarrow{\bw})\\
      &=\sum_{\overleftarrow{\bw}=\ba'\bb'}\anti(A)(\ful{\ba'}\bb')\anti(B)(\ba'\flr{\bb'})\\
      &=\sum_{\bw=\overleftarrow{\bb'}\overleftarrow{\ba'}}\anti(A)(\ful{\ba'}\bb')\anti(B)(\ba'\flr{\bb'})\\
      &=\sum_{\bw=\overleftarrow{\bb'}\overleftarrow{\ba'}}A(\overleftarrow{\ful{\ba'}\bb'})B(\overleftarrow{\ba'\flr{\bb'}})\\
      &=\sum_{\bw=\overleftarrow{\bb'}\overleftarrow{\ba'}}A((\overleftarrow{\bb'})\fur{\overleftarrow{\ba'}})B(\fll{\overleftarrow{\bb'}}(\overleftarrow{\ba'}))\\
      &=\sum_{\bw=\ba\bb}A(\ba\fur{\bb})B(\fll{\ba}\bb)\qquad (\overleftarrow{\bb'}\mapsto\ba,~\overleftarrow{\ba'}\mapsto\bb).
    \end{align}
    The rest statements are immediately shown by these two expressions.
\end{proof}

\begin{lemma}\label{lem:push_swamu_answamu}
  For bimoulds $A$ and $B$, we have
  \[\push(\swamu(A,B))=\answamu(\push(B),\push(A)).\]
\end{lemma}
\begin{proof}
  By the famous equality
  \begin{equation}
    \neg\circ\push=\anti\circ\swap\circ\anti\circ\swap,
  \end{equation}
  we have
  \begin{align}
    &\answamu(\push(B),\push(A))\\
    &=(\anti\circ\swap)(\mmu((\swap\circ\anti\circ\push)(B),(\swap\circ\anti\circ\push)(A)))\\
    &=(\anti\circ\swap)(\mmu((\neg\circ\anti\circ\swap)(B),(\neg\circ\anti\circ\swap)(A)))\\
    &=(\neg\circ\anti\circ\swap)(\mmu((\anti\circ\swap)(B),(\anti\circ\swap)(A)))\\
    &=(\push\circ\swap\circ\anti)(\mmu((\anti\circ\swap)(B),(\anti\circ\swap)(A)))\\
    &=(\push\circ\swap)(\mmu(\swap(A),\swap(B)))\\
    &=\push(\swamu(A,B)).
  \end{align}
\end{proof}

The senary relation \eqref{eq:senary} is nothing but the invariance under the map\footnote{It is originally denoted by $\fE\text{-}\push_{\ast}$. We changed the notation for more visiblity.} $\Esena$ defined as
\[\Esena\coloneqq\Eter^{-1}\circ\push\circ\mantar\circ\Eter\circ\mantar.\]
Actually, we can see that the $\Esena$-invariance and the $\Epush$-invariance are equivalent, by the following theorem.

\begin{theorem}[{\cite[(3.57)]{ecalle11}}]\label{thm:357}
    Let $B\in\ARI$.
    Then we have
    \begin{equation}\label{eq:357}
    (\id-\Esena)(B)=\swamu(\fes,(\id-\Epush)(B)).
    \end{equation}
\end{theorem}

The rest of this section is devoted to give a proof of this theorem.

\subsection{Another formulation of $\Epush$}
\begin{lemma}\label{lem:epush}
    Define
    \[\Eswap\coloneqq\adari(\fes)\circ\swap\circ\gaxit(\foz,\foz).\]
    Then we have
    \[\Epush=\neg\circ\mantar\circ\Eswap\circ\mantar\circ\swap.\]
\end{lemma}
\begin{proof}
    Using the left-right separation of $\gaxit$ (\cite[Corollary 3.11]{kawamura25}), we have
    \begin{align}
      \gaxit(\foz,\foz)\circ\mantar
      &=\ganit(\foz)\circ\gamit(\ganit(\foz)^{-1}(\foz))\circ\mantar\\
      &=\ganit(\foz)\circ\gamit(\fos)\circ\mantar\\
      &=\ganit(\foz)\circ\anti\circ\ganit(\anti(\fos))\circ\anti\circ\mantar\\
      &=-\ganit(\foz)\circ\anti\circ\ganit(\anti(\fos))\circ\pari\\
      &=-\ganit(\foz)\circ\anti\circ\pari\circ\ganit((\pari\circ\anti)(\fos))\\
      &=\ganit(\foz)\circ\mantar\circ\ganit(\invgani(\foz))\\
      &=\Omantar,
    \end{align}
    where we used the $\gantar$-invariance of $\fos$ (consequence of \cite[Propositions 5.3 (2), 2.14]{kawamura25}) and the fact that $\ganit((\pari\circ\anti)(\fos))$ gives the inverse of $\ganit(\foz)^{-1}$ (\cite[Proposition 5.3 (1)]{kawamura25}).
    Hence we obtain
    \begin{align}
      \Epush
      &=\Eneg\circ\mantar\circ\swap\circ\Omantar\circ\swap\\
      &=\neg\circ\adari(\fes)\circ\mantar\circ\swap\circ\gaxit(\foz,\foz)\circ\mantar\circ\swap\\
      &=\neg\circ\mantar\circ\adari(\fes)\circ\swap\circ\gaxit(\foz,\foz)\circ\mantar\circ\swap\\
      &=\neg\circ\mantar\circ\Eswap\circ\mantar\circ\swap
    \end{align}
    as desired, using the commutativity of $\adari(\fes)$ and $\mantar$, which is a consequence of the $\gantar$-invariance of $\fes$ and the fact that $\mantar$ preserves the $\ari$-bracket.
\end{proof}

\begin{lemma}\label{lem:preari_es}
    Let $B$ be an arbitrary element of $\ARI$.
    Then we have
    \[\preari(\fes,B)=\swamu(\fes,\mmu(\fes,B)-\answamu(\fes-1,B)).\]
\end{lemma}
\begin{proof}
    Using \cite[Remark 5.2]{kawamura25}, we have
    \begin{align}
        \amit(B)(\fes)(\bw)
        &=\sum_{\substack{\bw=\ba\bb\bc\\ \bb,\bc\neq\emp}}\fes(\ba\ful{\bb}\bc)B(\bb\flr{\bc})\\
        &=\sum_{\substack{\bw=\ba\bb\bc\\ \bb,\bc\neq\emp}}\fes(\ba\flr{\bc})\fes(\ful{\ba}\ful{\bb}\bc)B(\bb\flr{\bc})\\
        &=\sum_{\substack{\bw=\ba'\bc\\ \ba',\bc\neq\emp}}\mmu(\fes,B)(\ba'\flr{\bc})\fes(\ful{\ba'}\bc)\\
        &=\swamu(\fes,\mmu(\fes,B))(\bw)-\mmu(\fes,B)(\bw).
    \end{align}
    Similarly, it holds that
    \begin{align}
        \anit(B)(\fes)(\bw)
        &=\sum_{\substack{\bw=\ba\bb\bc\\ \ba,\bb\neq\emp}}\fes(\ba\fur{\bb}\bc)B(\fll{\ba}\bb)\\
        &=\sum_{\substack{\bw=\ba\bb\bc\\ \ba,\bb\neq\emp}}\fes(\ba\fur{\bb}\flr{\bc})\fes(\ful{\ba\fur{\bb}}\bc)B(\fll{\ba}\bb)\\
        &=\sum_{\substack{\bw=\ba\bb\bc\\ \ba,\bb\neq\emp}}\fes(\ba\fur{\bb}\flr{\bc})\fes(\ful{\ba\bb}\bc)B(\fll{\ba}\bb)\\
        &=\sum_{\substack{\bw=\ba'\bc\\ \ba'\neq\emp}}\fes(\ful{\ba'}\bc)\left(\answamu(\fes,B)(\ba'\flr{\bc})-B(\ba'\flr{\bc})\right)\\
        &=\swamu(\fes,\answamu(\fes-1,B))(\bw).
    \end{align}
    Therefore, by the definition of $\arit$ and $\preari$, we obtain
    \begin{align}
        \preari(\fes,B)
        &=\arit(B)(\fes)+\mmu(\fes,B)\\
        &=\amit(B)(\fes)-\anit(B)(\fes)+\mmu(\fes,B)\\
        &=\swamu(\fes,\mmu(\fes,B))-\swamu(\fes,\answamu(\fes-1,B))\\
        &=\swamu(\fes,\mmu(\fes,B)-\answamu(\fes-1,B)).
    \end{align}
\end{proof}

\begin{lemma}\label{lem:eswap}
    The inverse map of $\Eswap$ is given by
    \begin{align}
        B&\mapsto \mmu(\swap(\preari(\pari(\fes),B)),1+\fO)\\
        &=\mmu(\pari(\foz),\swap(\mmu(1+\pari(\fes),B)-\swap(\answamu(\pari(\fes),B))),1+\fO)\\
        &=\mmu(\pari(\foz),\mmu(\swap(B),1+\fO)+\swap(\answamu(\fE,B)-\mmu(\fE,B))).
    \end{align}
\end{lemma}
\begin{proof}
    First, let us show $\Eswap(\mmu(\swap(\preari(\pari(\fes),B)),1+\fO))=B$.
    Since the identity
    \begin{align}
        (\push\circ\swap\circ\invmu\circ\swap)(\foz)
        &=(\push\circ\swap\circ\invmu)(\fes)\\
        &=(\push\circ\swap\circ\push)(\fes)\\
        &=\swap(\fes)\\
        &=\foz
    \end{align}
    holds, the operator $\gaxit(\foz,\foz)$ coincides with $\girat(\foz)$.
    Therefore, from \cite[Remark 3.20]{kawamura25}, we obtain
    \begin{align}
        (\swap\circ\adari(\fes)^{-1})(B)
        &=(\swap\circ\adari(\pari(\fes)))(B)\\
        &=\swap(\gari(\preari(\pari(\fes),B),(\invgari\circ\pari)(\fes)))\\
        &=\swap(\gari(\preari(\pari(\fes),B),\fes))\\
        &=\gira(\swap(\preari(\pari(\fes),B)),\swap(\fes))\\
        &=\gira(\swap(\preari(\pari(\fes),B)),\foz)\\
        &=\mmu(\girat(\foz)(\swap(\preari(\pari(\fes),B))),\foz).
    \end{align}
    Using the left-right separation of $\gaxit$ (\cite[Corollary 3.11]{kawamura25}) and $\ganit(\foz)^{-1}(\foz)=\fos$, we see that
    \begin{align}
        \girat(\foz)^{-1}(\foz)
        &=\gaxit(\foz,\foz)^{-1}(\foz)\\
        &=(\gamit(\ganit(\foz)^{-1}(\foz))^{-1}\circ\ganit(\foz)^{-1})(\foz)\\
        &=\gamit(\fos)^{-1}(\fos)\\
        &=(\anti\circ\ganit(\anti(\fos))^{-1}\circ\anti)(\fos)\\
        &=(\anti\circ\pari\circ\ganit((\pari\circ\anti)(\fos))^{-1}\circ\pari\circ\anti)(\fos)\\
        &=(\anti\circ\pari\circ\ganit(\invgani(\foz))^{-1}\circ\pari\circ\anti)(\fos)\\
        &=(\anti\circ\pari\circ\ganit(\foz)\circ\pari\circ\anti)(\fos).
    \end{align}
    Then we can use the identity $\ganit(\foz)((\pari\circ\anti)(\fos))=1-\fO$ which follows from $\invgani(\foz)=(\pari\circ\anti)(\fos)$ and finally we get $\girat(\foz)^{-1}(\foz)=1+\fO$.
    Therefore we have
    \begin{align}
        (\swap\circ\adari(\fes)^{-1})(B)
        &=\mmu(\girat(\foz)(\swap(\preari(\pari(\fes),B))),\girat(\foz)(1+\fO))\\
        &=\girat(\foz)(\mmu(\swap(\preari(\pari(\fes),B)),1+\fO)),
    \end{align}
    which is equivalent to $\Eswap(\mmu(\swap(\preari(\pari(\fes),B)),1+\fO))=B$.
    This shows the first expression.\\

    Next we prove the second expression.
    From Lemma \ref{lem:preari_es} and the fact that $\pari$ is distributive for $\preari$, we have
    \begin{align}
        &\mmu(\swap(\preari(\pari(\fes),B)),1+\fO)\\
        &=\mmu((\swap\circ\pari)(\preari(\fes,\pari(B))),1+\fO)\\
        &=\mmu((\swap\circ\pari)(\swamu(\fes,\mmu(\fes,\pari(B))-\answamu(\fes-1,\pari(B)))),1+\fO).
    \end{align}
    Moreover, since $\pari$ is also distributive for $\mmu$, $\swamu$ and $\answamu$, we have
    \begin{align}
        &\mmu(\swap(\preari(\pari(\fes),B)),1+\fO)\\
        &=\mmu(\swap(\swamu(\pari(\fes),\mmu(\pari(\fes),B)-\answamu(\pari(\fes)-1,B))),1+\fO)\\
        &=\mmu(\mmu((\swap\circ\pari)(\fes),\swap(\mmu(\pari(\fes),B)-\answamu(\pari(\fes)-1,B))),1+\fO)\\
        &=\mmu(\pari(\foz),\swap(\mmu(\pari(\fes),B)-\answamu(\pari(\fes)-1,B)),1+\fO)\\
        &=\mmu(\pari(\foz),\swap(\mmu(1+\pari(\fes),B)-\answamu(\pari(\fes),B)),1+\fO).
    \end{align}
    Finally we show the third expression. It is sufficient to show
    \begin{equation}\label{eq:swup_third_proof}
        \swap(\mmu(\pari(\fes),B))-\answamu(\pari(\fes),B)=\mmu(\swap(\answamu(\fE,B)-\mmu(\fE,B)),\pari(\foz)).
    \end{equation}
    More finely, we actually show the following two identities:
    \begin{equation}\label{eq:swup_third_proof1}
        \swap(\mmu(1-\pari(\fes),B))
        =\mmu(\swap(\mmu(\fE,B)),\pari(\foz)),
    \end{equation}
    and
    \begin{equation}\label{eq:swup_third_proof2}
        \swap(B)-\answamu(\pari(\fes),B)=\mmu(\swap(\answamu(\fE,B)),\pari(\foz)).
    \end{equation}
    For \eqref{eq:swup_third_proof1}, applying Proposition \ref{prop:swamu_answamu} (1) to the right-hand side, we obtain
    \begin{align}
        \mmu(\swap(\mmu(\fE,B)),\pari(\foz))
        &=\mmu(\swamu(\fO,\swap(B)),\pari(\foz))\\
        &=\swamu(\mmu(\fO,\pari(\foz)),\swap(B))\\
        &=\swamu(\pari(\mmu(-\fO,\foz)),\swap(B))\\
        &=\swamu(\pari(\mmu(-\fO,\invmu(1-\fO))),\swap(B))\\
        &=\swap(B)-\swamu((\pari\circ\invmu)(1-\fO),\swap(B))\\
        &=\swamu(1-\pari(\foz),\swap(B))\\
        &=\swap(\mmu(1-\pari(\fes),B)).
    \end{align}
    Let us give a proof of \eqref{eq:swup_third_proof2}.
    Considering the image under $\swap$ of the right-hand side, we have
    \[\swap(\mmu(\swap(\answamu(\fE,B)),\pari(\foz)))
        =\swamu(\answamu(\fE,B),\pari(\fes)).\]
        Then we can apply Proposition \ref{prop:swamu_answamu} (3) as
    \begin{align}
        \swap(\mmu(\swap(\answamu(\fE,B)),\pari(\foz)))
        &=\answamu(\swamu(\fE,\pari(\fes)),B)\\
        &=\answamu(\swap(\mmu(\fO,\pari(\foz))),B).
    \end{align}
    We can also compute the part $\mmu(\fO,\pari(\foz))$ in the same way as previous arguments, so we have
    \begin{align}
        \swap(\mmu(\swap(\answamu(\fE,B)),\pari(\foz)))
        &=\answamu(\swap(1-\pari(\foz)),B)\\
        &=B-\answamu(\pari(\fes),B).
    \end{align}
\end{proof}

\begin{corollary}\label{cor:epush_explicit}
    For any $C\in\ARI$, we have
    \begin{multline}
      \Epush^{-1}(C)\\
      =\swap(\mmu(\mmu(1-\fO,(\push\circ\swap)(C))+\push(\mmu(\fO,\swap(C)))-\push(\swamu(\swap(C),\fO)),\foz)).
    \end{multline}
\end{corollary}
\begin{proof}
    From Lemma \ref{lem:epush} and the third expression of Lemma \ref{lem:eswap}, we have
    \begin{align}
        &(\swap\circ\Epush^{-1})(C)\\
        &=(\mantar\circ\Eswap^{-1}\circ\mantar\circ\neg)(C)\\
        &=\begin{multlined}[t]\mantar\Biggl(\mmu\biggl(\pari(\foz),\\\mmu\Bigl((\swap\circ\mantar\circ\neg)(C),1+\fO\Bigr)+\swap\Bigl(\answamu\bigl(\fE,(\mantar\circ\neg)(C)\bigr)\\
          -\mmu\bigl(\fE,(\mantar\circ\neg)(C)\bigr)\Bigr)\biggr)\Biggr)\end{multlined}\\
        &=\begin{multlined}[t]\mmu\Biggl(\mantar\biggl(\mmu\Bigl(\bigl(\swap\circ\mantar\circ\neg\bigr)(C),1+\fO\Bigr)\biggr)\\+(\mantar\circ\swap)\biggl(\answamu\Bigl(\fE,\bigl(\mantar\circ\neg\bigr)(C)\Bigr)-\mmu\Bigl(\fE,\bigl(\mantar\circ\neg\bigr)(C)\Bigr)\biggr),\foz\Biggr)\end{multlined}\\
        &=\begin{multlined}[t]
            \mmu\Biggl(\mmu\biggl(1-\fO,\bigl(\mantar\circ\swap\circ\mantar\circ\neg\bigr)(C)\biggr)\\+(\mantar\circ\swap\circ\anti\circ\swap)\biggl(\mmu\Bigl(\fO,\bigl(\swap\circ\anti\circ\mantar\circ\neg\bigr)(C)\Bigr)\biggr)\\
            -(\mantar\circ\swap)\biggl(\mmu\Bigl(\fE,\bigl(\mantar\circ\neg\bigr)(C)\Bigr)\biggr),\foz\Biggr)
        \end{multlined}.
    \end{align}
    Using the identity
    \[\mantar\circ\swap\circ\mantar\circ\swap=\neg\circ\push,\]
    we can deform the above calculation as
    \begin{align}
        &(\swap\circ\Epush^{-1})(C)\\
        &=\begin{multlined}[t]\mmu\biggl(\mmu\Bigl(1-\fO,(\push\circ\swap)(C)\Bigr)-(\pari\circ\neg\circ\push)\Bigl(\mmu(\fO,-(\swap\circ\pari\circ\neg)(C))\Bigr)\\
          -(\neg\circ\push\circ\swap\circ\mantar)\Bigl(\mmu(\fE,(\mantar\circ\neg)(C))\Bigr),\foz\biggr)\end{multlined}\\
        &=\mmu\Bigl(\mmu(1-\fO,(\push\circ\swap)(C))+\push(\mmu(\fO,\swap(C)))-(\push\circ\swap)(\mmu(C,\fE)),\foz\Bigr)\\
        &=\mmu\Bigl(\mmu(1-\fO,(\push\circ\swap)(C))+\push(\mmu(\fO,\swap(C)))-\push(\swamu(\swap(C),\fO)),\foz\Bigr).
    \end{align}
    This is what we wanted to show.
\end{proof}

\subsection{Another formulation of $\Esena$}
\begin{lemma}\label{lem:ter_explicit}
    For all $B\in\BIMU$, the following formulas are true.
    \begin{enumerate}
        \item $\Eter(B)=\mmu(B,1-\fE)+\answamu(B,\fE)$.
        \item $\Eter^{-1}(B)=\mmu(\answamu(B,\invmu(\fes)),\fes)$.
    \end{enumerate}
\end{lemma}
\begin{proof}
    The first assertion is a consequence of \eqref{eq:answamu}.  
  We prove that the inverse map of $B\mapsto\mmu(\answamu(B,\invmu(\fes)),\fes)$ is given by $\Eter$.
  That inverse becomes
  \begin{align}
    B\mapsto &\answamu(\mmu(B,\invmu(\fes)),(\anti\circ\swap\circ\invmu\circ\swap\circ\anti\circ\invmu)(\fes))\\
    &=\answamu(\mmu(B,(\pari\circ\anti)(\fes)),(\anti\circ\swap\circ\invmu\circ\swap\circ\pari)(\fes))\\
    &=\answamu(\mmu(B,(\pari\circ\anti)(\fes)),(\pari\circ\anti\circ\swap\circ\invmu)(\foz))\\
    &=\answamu(\mmu(B,(\pari\circ\anti)(\fes)),1+\fE)\\
    &=\answamu(\mmu(B,(\pari\circ\anti)(\fes)-1),1+\fE)+B+\answamu(B,\fE),
  \end{align}
  where we used the $\gantar$-invariance of $\fes$.
  So it suffices to show
  \[\answamu(\mmu(B,(\pari\circ\anti)(\fes)-1),1+\fE)=-\mmu(B,\fE),\]
  but this is deduced from Proposition \ref{prop:swamu_answamu} (2) as
  \begin{align}
    &\answamu(\mmu(B,(\pari\circ\anti)(\fes)-1),1+\fE)\\
    &=\mmu(B,\answamu((\pari\circ\anti)(\fes)-1,1+\fE))\\
    &=\mmu(B,(\anti\circ\swap)(\mmu((\swap\circ\anti\circ\pari\circ\anti)(\fes)-1,(\swap\circ\anti)(1+\fE))))\\
    &=\mmu(B,(\anti\circ\swap)(\mmu((\pari\circ\swap)(\fes)-1,1+\fO)))\\
    &=\mmu(B,(\anti\circ\swap)(\mmu(\invmu(1+\fO)-1,1+\fO)))\\
    &=\mmu(B,(\anti\circ\swap)(-\fO))\\
    &=-\mmu(B,\fE).
  \end{align}
\end{proof}

\begin{corollary}[{\cite[(3.56)]{ecalle11}}]
    The inverse map of $\Eter$ is explicitly given by the formula
    \[\Eter^{-1}(B)(\bw)=\sum_{\bw=\ba\bb\bc}B(\ba\fur{\bb})\invmu(\fes)(\fll{\ba}\bb)\fes(\bc),\]
    which is valid for any $B\in\BIMU$.
\end{corollary}
\begin{proof}
    This follows from Lemma \ref{lem:ter_explicit} and \eqref{eq:answamu}.
\end{proof}

\begin{corollary}
    For any $B\in\ARI$, the bimould $\Esena(B)$ has an explicit expression as
    \begin{equation}\label{eq:Esena_explicit}
    \Esena(B)
        =\mmu((\swap\circ\push^{-1})(\mmu(\foz,\swap(B+\mmu(\fE,B)-\swamu(B,\fE)))),\fes).
    \end{equation}
\end{corollary}
\begin{proof}
    Let $B$ be an arbitrary bimould in $\ARI$.
    From two equations in Lemma \ref{lem:ter_explicit}, we have
    \begin{align}
        &\Esena(B)\\
        &=(\Eter^{-1}\circ\push\circ\mantar\circ\Eter\circ\mantar)(B)\\
        &=\mmu(\answamu((\push\circ\mantar\circ\Eter\circ\mantar)(B),\invmu(\fes)),\fes)\\
        &=\begin{multlined}[t]\mmu\biggl((\anti\circ\swap)(\mmu((\swap\circ\anti\circ\push\circ\mantar\circ\Eter\circ\mantar)(B),\\(\swap\circ\anti)(\invmu(\fes)))),\fes\biggr)
        \end{multlined}
    \end{align}
    Using the identity
    \[\anti\circ\swap\circ\anti\circ\swap=\neg\circ\push,\]
    we see that the above computation becomes
    \begin{align}
        &\Esena(B)\\
        &=\begin{multlined}[t]
          \mmu\Biggl((\swap\circ\push^{-1}\circ\neg\circ\anti)\biggl(\mmu\Bigl((\neg\circ\anti\circ\swap\circ\mantar\circ\Eter\circ\mantar)(B),\\
          (\swap\circ\anti)(\invmu(\fes))\Bigr)\biggr),\fes\Biggr)
        \end{multlined}\\
        &=\begin{multlined}[t]
          \mmu\biggl((\swap\circ\push^{-1})(\mmu((\neg\circ\anti\circ\swap\circ\anti)(\invmu(\fes)),\\
          (\swap\circ\mantar\circ\Eter\circ\mantar)(B))),\fes\biggr).
        \end{multlined}
    \end{align}
    The factor $(\neg\circ\anti\circ\swap\circ\anti)(\invmu(\fes))$ appearing above coincides with $\foz$ due to \cite[Proposition 5.3 (1)]{kawamura25}.
    We can compute the part above which depends on $B$ as
    \begin{align}
        &(\mantar\circ\Eter\circ\mantar)(B)\\
        &=\mantar(\mantar(B)-\mmu(\mantar(B),\fE)+\answamu(\mantar(B),\fE))\\
        &=B+\mmu(\fE,B)+(\mantar\circ\anti)(\swamu((\anti\circ\mantar)(B),\anti(\fE)))\\
        &=B+\mmu(\fE,B)+\pari(\swamu(\pari(B),\fE))\\
        &=B+\mmu(\fE,B)-\swamu(B,\fE).
    \end{align}
    Summarizing these computation, we obtain \eqref{eq:Esena_explicit}.
\end{proof}

\subsection{Proof of Theorem \ref{thm:357}}
\begin{lemma}
    Theorem \ref{thm:357} is equivalent to that, for all $B\in\ARI$ the identity
    \begin{equation}\label{eq:357rephrase}
        \ORush(\mmu(\fO,B)+\mmu(1-\fO,(\swap\circ\Esena\circ\swap)(B)))=\mmu(B,1-\fO)
\end{equation}
    holds, where the operator $\ORush$ is defined as
    \[X\mapsto\mmu(1-\fO,\push(X))+\push(\mmu(\fO,X))-\push(\swamu(X,\fO)).\]
\end{lemma}
\begin{proof}
    Since the inverse about $\swamu$ is given by $\swap\circ\invmu\circ\swap$, we see that \eqref{eq:357} is equivalent to
    \[\swamu(1-\fE,B-\Esena(B))=B-\Epush(B),\]
    which is moreover equivalent to
    \[(\swap\circ\Epush^{-1})(B-\swamu(1-\fE,B-\Esena(B)))=\swap(B).\]
    Using Lemma \ref{lem:epush}, we rewrite the left-hand side above as
    \begin{align}
        &(\swap\circ\Epush^{-1})(B-\swamu(1-\fE,B-\Esena(B)))\\
        &=(\swap\circ\Epush^{-1})(\swamu(\fE,B)+\swamu(1-\fE,\Esena(B)))\\
        &=(\swap\circ\Epush^{-1}\circ\swap)(\mmu(\fO,\swap(B))+\mmu(1-\fO,(\swap\circ\Esena)(B))).
    \end{align}
    Thus it suffices to show that, for any $C\in\ARI$ the identity
    \begin{equation}\label{eq:357rephrase_rest}
        \mmu((\swap\circ\Epush^{-1}\circ\swap)(C),1-\fO)=\ORush(C)
    \end{equation}
    holds.
    Using the definition of $\ORush$, we can rewrite the statement of Corollary \ref{cor:epush_explicit} as
    \[(\swap\circ\Epush^{-1})(C)=\mmu((\ORush\circ\swap)(C),\foz).\]
    By replacing $C$ by $\swap(C)$ here, the identity \eqref{eq:357rephrase_rest} immediately follows.
\end{proof}

\begin{proposition}
    The identity \eqref{eq:357rephrase} is valid for any $B\in\ARI$, and hence Theorem \ref{thm:357} holds.
\end{proposition}
\begin{proof}
  Throughout this proof, we write
  \[B'\coloneqq B-\mmu(B,\fO)+\swamu(\fO,B)\]
  for simplicity.
    Put
    \begin{align}
        R_{1}(X)&\coloneqq\push(X),\\
        R_{2}(X)&\coloneqq\mmu(\fO,\push(X)),\\
        R_{3}(X)&\coloneqq\push(\mmu(\fO,X))
    \end{align}
    and
    \[R_{4}(X)\coloneqq\push(\swamu(X,\fO))\]
    for any bimould $X$, so that $\sum_{i=1}^{4}(-1)^{i-1}R_{i}=\ORush$ holds.
    Note that Lemma \ref{lem:push_swamu_answamu} shows another expression
    \[R_{4}(X)=-\answamu(\fO,\push(X)).\]
    As our goal is to show \eqref{eq:357rephrase},
    we furthermore consider a decomposition of the term $(\swap\circ\Esena\circ\swap)(B)$: first remember the expression
    \[(\swap\circ\Esena\circ\swap)(B)=\swamu(\push^{-1}(\mmu(\foz,B')),\foz),\]
    which is an easy corollary of \eqref{eq:Esena_explicit}.
    Then we put
    \begin{align}
        A_{1}&\coloneqq \swamu(\push^{-1}(\mmu(\foz,B')),\mmu(\fO,\foz)),\\
        A_{11}&\coloneqq \swamu(\push^{-1}(\mmu(\foz,B')),\mmu(\fO,\fO,\foz)),\\
        A_{12}&\coloneqq \swamu(\push^{-1}(\mmu(\foz,B')),\fO),\\
        A_{2}&\coloneqq \push^{-1}(\mmu(\foz,B')),\\
        A_{21}&\coloneqq \push^{-1}(\mmu(\foz-1,B')),\\
    \end{align}
    and
    \[A_{22}\coloneqq\push^{-1}(B'),\]
    so that $A_{1}=A_{11}+A_{12}$, $A_{2}=A_{21}+A_{22}$ and
    \[A_{1}+A_{2}=A_{11}+A_{12}+A_{21}+A_{22}=(\swap\circ\Esena\circ\swap)(B)\]
    due to the fact $\mmu(\fO,\foz)=\foz-1$.
    By definition, we can check that
    \begin{equation}\label{eq:3_1_2}
      R_{1}(\mmu(\fO,A_{1}+A_{2})) = R_{3}(A_{1}+A_{2})
    \end{equation}
    Moreover we have
    \[ R_{1}(A_{21}) = R_{2}(A_{2}), \]
    by the identity $\mmu(\fO,\foz)=\foz-1$, and
    \[ R_{1}(A_{12}) = R_{4}(A_{2}),\]
    by Lemma \ref{lem:push_swamu_answamu}.
    These equations is combined as
    \begin{equation}\label{eq:24_2}
      R_{1}(A_{12}+A_{21})=(R_{2}+R_{4})(A_{2}),
    \end{equation}
    and it shows that, together with \eqref{eq:3_1_2},
    \begin{equation}\label{eq:first_step}
      (-R_{2}+R_{3}-R_{4})(A_{1}+A_{2})-R_{1}(\mmu(\fO,A_{1}+A_{2}))
      =(-R_{2}-R_{4})(A_{1})-R_{1}(A_{12}+A_{21}).
    \end{equation}
    On the other hand, Lemma \ref{lem:push_swamu_answamu} also shows
    \begin{align}
      R_{1}(A_{11})-R_{4}(A_{1})
      &=\begin{multlined}[t]
        \push(\swamu(\push^{-1}(\mmu(\foz,B')),\mmu(\fO,\fO,\foz)))\\
        +\answamu(\fO,\push(\swamu(\push^{-1}(\mmu(\foz,B')),\mmu(\fO,\foz))))
      \end{multlined}\\
      &=\begin{multlined}[t]
        \answamu(\push(\mmu(\fO,\fO,\foz)),\mmu(\foz,B'))\\
        +\answamu(\fO,\answamu(\push(\mmu(\fO,\foz)),\mmu(\foz,B')))
      \end{multlined}\\
      &=\begin{multlined}[t]
        -\answamu(\swamu(\fO,\mmu(\fO,\foz)),\mmu(\foz,B'))\\
        -\answamu(\fO,\answamu(\swamu(\fO,\foz),\mmu(\foz,B'))),
      \end{multlined}
    \end{align}
    where we used the relation $\push(\fO,M)=-\swamu(\fO,M)$ which holds for any bimould $M$.
    Hence we obtain
    \begin{align}
      &(R_{1}(A_{11})-R_{4}(A_{1}))(w_{1},\ldots,w_{r})\\
      &=\begin{multlined}[t]-\sum_{w_{1}\cdots w_{r}=\ba\bb}\swamu(\fO,\mmu(\fO,\foz))(\ba\fur{\bb})\mmu(\foz,B')(\fll{\ba}\bb)\\
        -\fO(w_{1}\fur{w_{2}\cdots w_{r}})\sum_{\fll{w_{1}}(w_{2}\cdots w_{r})=\ba'\bb}\swamu(\fO,\foz)(\ba'\fur{\bb})\mmu(\foz,B')(\fll{\ba'}\bb)\end{multlined}\\
        &=\begin{multlined}[t]-\sum_{w_{2}\cdots w_{r}=\ba'w_{i}\bb}\swamu(\fO,\mmu(\fO,\foz))(w_{1}\ba'(w_{i}\fur{\bb}))\mmu(\foz,B')(\fll{w_{i}}\bb)\\
        -\fO(w_{1}\fur{w_{2}\cdots w_{r}})\sum_{w_{2}\cdots w_{r}=\ba'w_{i}\bb}\swamu(\fO,\foz)(\fll{w_{1}}(\ba'w_{i}\fur{\bb}))\mmu(\foz,B')(\fll{w_{i}}\bb)\end{multlined}\\
        &=\begin{multlined}[t]-\sum_{w_{2}\cdots w_{r}=\ba'w_{i}\bb}\fO(\ful{w_{1}\ba'}w_{i}\fur{\bb})\fO(w_{1}\flr{w_{i}})\foz(\ba'\flr{w_{i}})\mmu(\foz,B')(\fll{w_{i}}\bb)\\
        -\fO(w_{1}\fur{w_{2}\cdots w_{r}})\sum_{w_{2}\cdots w_{r}=\ba'w_{i}\bb}\fO(\fll{w_{1}}\ful{\ba'}w_{i}\fur{\bb})\foz(\ba'\fur{w_{i}})\mmu(\foz,B')(\fll{w_{i}}\bb)\end{multlined}\\
        &=-\sum_{w_{2}\cdots w_{r}=\ba'w_{i}\bb}\left(\fO(\ful{w_{1}}\ful{\ba'}w_{i}\fur{\bb})\fO(w_{1}\flr{w_{i}})+\fO(w_{1}\fur{w_{2}\cdots w_{r}})\fO(\fll{w_{1}}\ful{\ba'}w_{i}\fur{\bb})\right)
          \foz(\ba'\flr{w_{i}})\mmu(\foz,B')(\fll{w_{i}}\bb)
    \end{align}
    due to Proposition \ref{prop:swamu_answamu}.
    Applying the tripartite relation, we have
    \begin{align}
      &(R_{1}(A_{11})-R_{4}(A_{1}))(w_{1},\ldots,w_{r})\\
      &=-\sum_{w_{2}\cdots w_{r}=\ba'w_{i}\bb}\fO(w_{1})\fO(\ful{\ba'}w_{i}\fur{\bb})\foz(\ba'\flr{w_{i}})\mmu(\foz,B')(\fll{w_{i}}\bb)\\
      &=-\fO(w_{1})\sum_{w_{2}\cdots w_{r}=\ba'w_{i}\bb}\swamu(\fO,\foz)(\ba'w_{i}\fur{\bb})\mmu(\foz,B')(\fll{w_{i}}\bb)\\
      &=-\fO(w_{1})\answamu(\swamu(\fO,\foz),\mmu(\foz,B'))(w_{2},\ldots,w_{r})\\
      &=\mmu(\fO,\answamu(\push(\mmu(\fO,\foz)),\mmu(\foz,B')))(w_{1},\ldots,w_{r})\\
      &=\mmu(\fO,\push(\swamu(\push^{-1}(\mmu(\foz,B')),\mu(\fO,\foz))))(w_{1},\ldots,w_{r})\\
    \end{align}
    and thus
    \begin{equation}
      R_{1}(A_{11})-R_{4}(A_{1})=R_{2}(A_{1}).
    \end{equation}
    Using this identity and \eqref{eq:first_step}, we deduce that
    \begin{align}\label{eq:second_step}
      \begin{split}
      &(-R_{2}+R_{3}-R_{4})(A_{1}+A_{2})-R_{1}(\mmu(\fO,A_{1}+A_{2}))\\
      &=(-R_{2}-R_{4})(A_{1})-R_{1}(A_{12}+A_{21})\\
      &=-R_{1}(A_{11})-R_{1}(A_{12}+A_{21})\\
      &=-R_{1}(A_{1}+A_{2}-A_{22}).
      \end{split}
    \end{align}
    Furthermore, for any bimould $M$, we have
    \begin{align}
      &(-R_{2}+R_{3}-R_{4})(\mmu(\fO,M))(w_{1},\ldots,w_{r})\\
      &=(-\mmu(\fO,\push(\mmu(\fO,M)))+\push(\mmu(\fO,\fO,M))+\answamu(\fO,\push(\mmu(\fO,M))))(w_{1},\ldots,w_{r})\\
      &=(\mmu(\fO,\swamu(\fO,M))-\swamu(\fO,\mmu(\fO,M))-\answamu(\fO,\swamu(\fO,M)))(w_{1},\ldots,w_{r})\\
      &=\begin{multlined}[t]\fO(w_{1})\fO(\ful{w_{2}\cdots w_{r-1}}w_{r})M((w_{2}\cdots w_{r-1})\flr{w_{r}})-\fO(\ful{w_{1}\cdots w_{r-1}}w_{r})\fO(w_{1}\flr{w_{r}})M((w_{2}\cdots w_{r-1})\flr{w_{r}})\\-\fO(w_{1}\fur{w_{2}\cdots w_{r}})\fO(\fll{w_{1}}\ful{w_{2}\cdots w_{r-1}}w_{r})M((w_{2}\cdots w_{r-1})\flr{w_{r}})\end{multlined}\\
      &=\begin{multlined}[t](\fO(w_{1})\fO(\ful{w_{2}\cdots w_{r-1}}w_{r})-\fO(\ful{w_{1}\cdots w_{r-1}}w_{r})\fO(w_{1}\flr{w_{r}})-\fO(w_{1}\fur{w_{2}\cdots w_{r}})\fO(\fll{w_{1}}\ful{w_{2}\cdots w_{r-1}}w_{r}))\\\cdot M((w_{2}\cdots w_{r-1})\flr{w_{r}}),\end{multlined}
    \end{align}
    which becomes $0$ by the tripartite identity.
    In particular, one has
    \begin{equation}
      (-R_{2}+R_{3}-R_{4})(\mmu(\fO,B))=0
    \end{equation}
    and 
    \begin{equation}
      (-R_{2}+R_{3}-R_{4})(\mmu(\fO,A_{1}+A_{2}))=0.
    \end{equation}
    From these two identities, we show
    \begin{align}
      &\ORush(\mmu(\fO,B)+\mmu(1-\fO,(\swap\circ\Esena\circ\swap)(B)))\\
      &=(R_{1}-R_{2}+R_{3}-R_{4})(\mmu(\fO,B)+\mmu(1-\fO,A_{1}+A_{2}))\\
      &=R_{1}(\mmu(\fO,B))+(R_{1}-R_{2}+R_{3}-R_{4})(A_{1}+A_{2})-R_{1}(\mmu(\fO,A_{1}+A_{2})).
    \end{align}\\
    Applying \eqref{eq:second_step} here, finally we get
    \begin{align}
      &\ORush(\mmu(\fO,B)+\mmu(1-\fO,(\swap\circ\Esena\circ\swap)(B)))\\
      &=R_{1}(\mmu(\fO,B))+R_{1}(A_{22})\\
      &=\push(\mmu(\fO,B))+B'\\
      &=-\swamu(\fO,B)+B-\mmu(B,\fO)+\swamu(\fO,B)\\
      &=\mmu(B,1-\fO).
    \end{align}
    which exactly shows \eqref{eq:357rephrase}.
\end{proof}

\section{From $\push$-invariants to senary relation}
This section is devoted to give a proof of Theorem \ref{thm:push_senary}.

\begin{lemma}\label{lem:irat_mantar}
    For every $X\in\LU$, we have
    \[\irat(\mantar(X))\circ\mantar=\mantar\circ\irat(\push^{-1}(X)).\]
\end{lemma}
\begin{proof}
    It immediately follows from the expression $\irat(X)=\axit(X,-\push(X))$.
\end{proof}
\begin{lemma}\label{lem:axit_mantar}
    Let $X\in\MU$ be a $\gantar$-invariant bimould and $A,B\in\LU$,
    Then we have
    \[\axit(A,B)(X)=\mmu(X,(\axit(A,B)\circ\mantar)(X),X).\]
\end{lemma}
\begin{proof}
    Since $\axit(A,B)$ is a $\mmu$-derivation (\cite[Proposition 2.2.1]{schneps15}), we see that
    \begin{align}
        0&=\axit(A,B)(1)\\
        &=\axit(A,B)(\mmu(X,\invmu(X)))\\
        &=\mmu\left(\axit(A,B)(X),\invmu(X)\right)+\mmu(X,(\axit(A,B)\circ\invmu)(X))
    \end{align}
    holds.
    Since $\mantar(X)=-\invmu(X)$ due to the $\gantar$-invariance of $X$, one has
    \begin{align}
        \axit(A,B)(X)
        &=\mmu(-\mmu(X,(\axit(A,B)\circ\invmu)(X)),X)\\
        &=\mmu(X,-(\axit(A,B)\circ\invmu)(X),X)\\
        &=\mmu(X,(\axit(A,B)\circ\mantar)(X),X).
    \end{align}
\end{proof}

\begin{lemma}\label{lem:garit_mantar}
    When $X\in\MU$ is $\gantar$-invariant, the operator $\garit(X)$ commutes with $\mantar$.
\end{lemma}
\begin{proof}
    First we remember two identities,
    \[\anti\circ\garit(\anti(Y))\circ\anti=\garit(\invmu(Y))\]
    and
    \[\pari\circ\garit(\pari(Y))\circ\pari=\garit(Y),\]
    that are valid for any $Y\in\MU$.
    They yield
    \begin{align}
        \garit(X)\circ\mantar
        &=-\garit(X)\circ\pari\circ\anti\\
        &=-\pari\circ\garit(\pari(X))\circ\anti\\
        &=\mantar\circ\garit(\gantar(X)),
    \end{align}
    and then we can apply the assumption $\gantar(X)=X$.
\end{proof}

\begin{lemma}\label{lem:garit_pil}
    We have
    \[(\garit(\foss)\circ\invgari)(\foss)=\invmu(\foss)\]
    and
    \[(\garit(\foss)\circ\mantar\circ\invgari)(\foss)=-\foss.\]
\end{lemma}
\begin{proof}
    The first identity directly comes from the definition of $\invgari$.
    The second follows from the first one and Lemma \ref{lem:garit_mantar}.
\end{proof}

\begin{corollary}\label{cor:first}
    For every $A\in\LU$ we have
    \[(\ganit(\foz)^{-1}\circ\swap\circ\adari(\dess))(A)=\fragari(\preira(\foss,\swap(A)),\foss).\]
\end{corollary}
\begin{proof}
    One has
    \begin{align}
        (\swap\circ\adari(\dess))(A)
        &=\swap(\fragari(\preari(\dess,A),\dess))\\
        &=\fragira(\swap(\preari(\dess,A)),\foss)
    \end{align}
    by definition.
    Then we apply the first fundamental identity \eqref{eq:first_fundamental} as
    \begin{align}(\swap\circ\adari(\dess))(A)&=\ganit(\foz)(\fragari(\swap(\preari(\dess,A)),\foss))\\&=\ganit(\foz)(\fragari(\preira(\foss,\swap(A)),\foss)).\end{align}
\end{proof}

\begin{theorem}\label{thm:komiyamanote}
    The identity
    \[\swamu\left(\dess,((\id-\push^{-1})\circ\adari(\dess)^{-1})(B)\right)=\gari((\id-\Epush^{-1})(B),\dess)\]
    holds for any $B\in\LU$.
\end{theorem}
\begin{proof}
Since the left-hand side becomes
\begin{align}
    &\swamu\left(\dess,((\id-\push^{-1})\circ\adari(\dess)^{-1})(B)\right)\\
    &=\swap(\mmu(\foss,(\swap\circ(\id-\push^{-1})\circ\adari(\dess)^{-1})(B)))\\
    &=\begin{multlined}[t]\swap(\preira(\foss,(\swap\circ(\id-\push^{-1})\circ\adari(\dess)^{-1})(B)))\\
    -(\swap\circ\irat((\swap\circ(\id-\push^{-1})\circ\adari(\dess)^{-1})(B)))(\foss)
\end{multlined}\\
    &=\begin{multlined}[t]\preari(\dess,((\id-\push^{-1})\circ\adari(\dess)^{-1})(B))\\
        -(\swap\circ\irat((\swap\circ(\id-\push^{-1})\circ\adari(\dess)^{-1})(B)))(\foss),
    \end{multlined}
\end{align}
it is sufficient to show
\begin{multline}
    \gari((\id-\fE\text{-}\push^{-1})(B),\dess)\\
    =\preari(\dess,((\id-\push^{-1})\circ\adari(\dess)^{-1})(B))\\
    -(\swap\circ\irat((\swap\circ(\id-\push^{-1})\circ\adari(\dess)^{-1})(B)))(\foss)
\end{multline}
which is also equivalent to
\begin{multline}\label{eq:925}
        (\id-\fE\text{-}\push^{-1})(B)\\
        =\fragari(\preari(\dess,((\id-\push^{-1})\circ\adari(\dess)^{-1})(B)),\dess)\\
    -\fragari((\swap\circ\irat((\swap\circ(\id-\push^{-1})\circ\adari(\dess)^{-1})(B)))(\foss),\dess)
\end{multline}
Sicne the first term in the right-hand side of \eqref{eq:925} is computed as
\begin{align}
    &\fragari(\preari(\dess,((\id-\push^{-1})\circ\adari(\dess)^{-1})(B)),\dess)\\
    &=(\adari(\dess)\circ(\id-\push^{-1})\circ\adari(\dess)^{-1})(B)\\
    &=B-(\adari(\dess)\circ\push^{-1}\circ\adari(\dess)^{-1})(B),
\end{align}
The equality \eqref{eq:925} is rephrased as
\begin{align}\label{eq:goal}
    \begin{split}
    &\fE\text{-}\push^{-1}(B)-(\adari(\dess)\circ\push^{-1}\circ\adari(\dess)^{-1})(B)\\
    &=\fragari((\swap\circ\irat((\swap\circ(\id-\push^{-1})\circ\adari(\dess)^{-1})(B)))(\foss),\dess).
    \end{split}
\end{align}
We give a proof of \eqref{eq:goal} by computing the right-hand side.
Denote $A\coloneqq(\push^{-1}\circ\adari(\dess)^{-1})(B)$.
Since
\begin{multline}
    \fragari((\swap\circ\irat((\swap\circ(\push-\id))(A)))(\foss),\dess)\\
    =\swap(\fragira(\irat((\swap\circ(\push-\id))(A))(\foss),\foss))
\end{multline}
holds, applying \eqref{eq:first_fundamental}, we obtain
\begin{align}\label{eq:lhs}
    \begin{split}
    &\fragari((\swap\circ\irat((\swap\circ(\push-\id))(A)))(\foss),\dess)\\
    &=\swap(\fragira(\irat((\swap\circ(\push-\id))(A))(\foss),\foss))\\
    &=(\swap\circ\ganit(\foz))(\fragari(\irat((\swap\circ(\push-\id))(A))(\foss),\foss)).
    \end{split}
\end{align}
From the identity $(\swap\circ\push)^{\circ 2}=\id$ and Lemma \ref{lem:irat_mantar}, it follows that
\begin{align}
    &\irat((\swap\circ(\push-\id))(A))(\foss)\\
    &=\irat(((\push^{-1}\circ\swap)-\swap)(A))(\foss)\\
    &=(\mantar\circ\irat((\mantar\circ\swap)(A))\circ\mantar-\irat(\swap(A)))(\foss).
\end{align}
Next, we can apply Lemma \ref{lem:axit_mantar} as
\begin{align}
    &\irat((\swap\circ(\push-\id))(A))(\foss)\\
    &=\mantar(\mmu(\mantar(\foss),\irat((\mantar\circ\swap)(A))(\foss),\mantar(\foss)))-(\irat(\swap(A)))(\foss).
\end{align}
Therefore the computation proceeds as
\begin{align}
    &\mmu(\foss,(\mantar\circ\irat((\swap\circ(\push-\id))(A)))(\foss))\\
    &=\mmu(\irat((\mantar\circ\swap)(A))(\foss),(\pari\circ\anti)(\foss))-\mmu(\foss,(\mantar\circ\irat(\swap(A)))(\foss))\\
    &=\begin{multlined}[t]
        \mmu(\preira(\foss,(\mantar\circ\swap)(A))-\mmu(\foss,(\mantar\circ\swap)(A)),(\pari\circ\anti)(\foss))\\
        -\mmu(\foss,\mantar(\preira(\foss,\swap(A)))-\mantar(\mmu(\foss,\swap(A))))
    \end{multlined}\\
    &=\begin{multlined}[t]
        \mmu(\preira(\foss,(\mantar\circ\swap)(A)),(\pari\circ\anti)(\foss))\\-\mmu(\foss,\mantar(\preira(\foss,\swap(A))))
        -\mmu(\foss,(\mantar\circ\swap)(A),(\pari\circ\anti)(\foss))\\+\mmu(\foss,(\mantar\circ\swap)(A),(\pari\circ\anti)(\foss))
    \end{multlined}\\
    &=\mmu(\preira(\foss,(\mantar\circ\swap)(A)),(\pari\circ\anti)(\foss))-\mmu(\foss,\mantar(\preira(\foss,\swap(A)))).
\end{align}
Using Lemma \ref{lem:garit_pil} and $(\pari\circ\anti)(\foss)=\invmu(\foss)$, we have
\begin{align}
    &\mmu(\foss,(\mantar\circ\irat((\swap\circ(\push-\id))(A)))(\foss))\\
    &=\begin{multlined}[t]
        \mmu(\preira(\foss,(\mantar\circ\swap)(A)),(\garit(\foss)\circ\invgari)(\foss))\\
        +\mmu((\garit(\foss)\circ\mantar\circ\invgari)(\foss),\mantar(\preira(\foss,\swap(A))))
    \end{multlined}\\
    &=\begin{multlined}[t]
        \garit(\foss)(\mmu(\garit(\foss)^{-1}(\preira(\foss,(\mantar\circ\swap)(A))),\invgari(\foss)))\\
        +\garit(\foss)(\mmu((\mantar\circ\invgari)(\foss),(\garit(\foss)^{-1}\circ\mantar)(\preira(\foss,\swap(A)))))
    \end{multlined}\\
    &=\begin{multlined}[t]\garit(\foss)(\mmu(\garit(\foss)^{-1}(\preira(\foss,(\mantar\circ\swap)(A))),\invgari(\foss)))\\
        -(\garit(\foss)\circ\mantar)(\mmu(\garit(\foss)^{-1}(\preira(\foss,\swap(A))),\invgari(\foss)))
    \end{multlined}\\
    &=\begin{multlined}[t]\garit(\foss)(\fragari(\preira(\foss,(\mantar\circ\swap)(A)),\foss)\\-\mantar(\fragari(\preira(\foss,\swap(A)),\foss))).\end{multlined}
\end{align}
Applying Corollary \ref{cor:first}, we see that it moreover becomes
\begin{align}\label{eq:9252}
    \begin{split}
    &\garit(\foss)^{-1}(\mmu(\foss,(\mantar\circ\irat((\swap\circ(\push-\id))(A)))(\foss)))\\
    &=(\ganit(\foz)^{-1}\circ\swap\circ\adari(\dess)\circ\swap\circ\mantar\circ\swap)(A)-(\mantar\circ\ganit(\foz)^{-1}\circ\swap\circ\adari(\dess))(A)\\
    &=(\ganit(\foz)^{-1}\circ\swap)((\adari(\dess)\circ\mantar\circ\neg\circ\push)(A)-(\swap\circ\Omantar\circ\swap\circ\adari(\dess))(A))\\
    &=(\ganit(\foz)^{-1}\circ\swap\circ\fE\text{-}\neg\circ\mantar)((\adari(\dess)\circ\push)(A)-(\fE\text{-}\push\circ\adari(\dess))(A)).
    \end{split}
\end{align}
Here we used in the last equality, the commutativity between $\mantar$ and $\adari(X)$, which is true for any $X\in\GARI_{\as}$, and the identity $\Eneg=\adari(\dess)\circ\neg\circ\adari(\dess)^{-1}$, which follows from \cite[Theorem 1.3]{kawamura25}.
On the other hand, Lemmas \ref{lem:garit_pil} and \ref{lem:garit_mantar} yield
\begin{align}
    &\garit(\foss)^{-1}(\mmu(\foss,(\mantar\circ\irat((\swap\circ(\push-\id))(A)))(\foss)))\\
    &=\mmu((\pari\circ\anti\circ\invgari)(\foss),(\garit(\foss)^{-1}\circ\mantar\circ\irat((\swap\circ(\push-\id))(A)))(\foss))\\
    &=\mmu((\pari\circ\anti\circ\invgari)(\foss),(\mantar\circ\garit(\foss)^{-1}\circ\irat((\swap\circ(\push-\id))(A)))(\foss))\\
    &=\mantar(\mmu((\garit(\foss)^{-1}\circ\irat((\swap\circ(\push-\id))(A)))(\foss),\invgari(\foss)))\\
    &=\mantar(\fragari(\irat((\swap\circ(\push-\id))(A))(\foss),\foss)).
\end{align}
By combining this and \eqref{eq:9252}, we arrive at the equality
\begin{multline}
    (\fE\text{-}\neg\circ\mantar\circ\swap\circ\ganit(\foz)\circ\mantar)(\fragari(\irat((\swap\circ(\push-\id))(A))(\foss),\foss))\\
    =(\adari(\dess)\circ\push)(A)-(\fE\text{-}\push\circ\adari(\dess))(A).
\end{multline}
Since this identity and \eqref{eq:lhs} shows
\begin{multline}
    \fE\text{-}\push(\fragari((\swap\circ\irat((\swap\circ(\push-\id))(A)))(\foss),\dess))\\
    =(\adari(\dess)\circ\push)(A)-(\fE\text{-}\push\circ\adari(\dess))(A),
\end{multline}
we see that our goal \eqref{eq:goal} amounts to
\begin{multline}
    \fE\text{-}\push^{-1}(B)-(\adari(\dess)\circ\push^{-1}\circ\adari(\dess)^{-1})(B)\\=\fE\text{-}\push^{-1}((\adari(\dess)\circ\push)(A)-(\fE\text{-}\push\circ\adari(\dess))(A))
\end{multline}
This identity is easy by the definition $(\push^{-1}\circ\adari(\dess)^{-1})(B)=A$.
\end{proof}

\begin{corollary}[$=$ Theorem \ref{thm:push_senary}]\label{cor:push_senary}
    Both maps $\adari(\fess)$ and $\adari(\dess)$ give an $R$-linear isomorphism $\ARI_{\push}\to\ARI_{\Esena}$.
\end{corollary}
\begin{proof}
    From Theorem \ref{thm:357}, we know that the module of all $\Epush$-invariants is equal to the module of all solution of the senary relation, namely $\ARI_{\fE\text{-}\push}=\ARI_{\Esena}$. So we shall show $\ARI_{\push}\simeq\ARI_{\fE\text{-}\push}$.\\

    First we check that each image satisfies the senary relation.
    Writing $A\coloneqq(\push^{-1}\circ\adari(\dess)^{-1})(B)$ in Theorem \ref{thm:komiyamanote}, we see that for all $A\in\LU$ the identity
    \[\swamu(\dess,(\push-\id)(A))=\gari(((\id-\fE\text{-}\push^{-1})\circ\adari(\dess)\circ\push)(A),\dess)\]
    holds. Assume $A\in\ARI_{\push}$. Then the above identity becomes
    \[0=\gari(((\id-\fE\text{-}\push^{-1})\circ\adari(\dess))(A),\dess),\]
    so that $\adari(\dess)(A)$ should be $\fE\text{-}\push$-invariant.\\

    It is an easy fact that each image under $\adari(\dess)^{-1}$ is $\push$-invariant, due to Theorem \ref{thm:komiyamanote}.
    Note that, every argument we did from Lemma \ref{lem:garit_pil} to here remains available when we replace $(\foss,\dess)$ by $(\doss,\fess)$. 
\end{proof}

\appendix
\section{Bisymmetrality of secondary bimoulds}\label{sec:bisymmetrality}
In this appendix, we show that the secondary bimould $\fess$ is bisymmetral, namely both $\fess$ and $\dess=\swap(\fess)$ are symmetral. Since the ``non-swap'' side, $\fess$, is already shown to be symmetral in \cite[Proposition 6.13]{kawamura25} in a more general form, we check the symmetrality of $\dess$.
Noting that $\gantar$-invariance is a consequence of symmetrality (\cite[Proposition 2.14]{kawamura25}), we can say that our plan to prove $\dess\in\GARI_{\as}$ is an enhancement of arguments used to show $\gantar(\dess)=\dess$ in \cite[Theorem 7.9]{kawamura25}.
\begin{remark}
  As shown below, our proof is based on \'{E}calle's \emph{theory of dilators}.
  Especially, the ones we shall use are called \emph{$\gari$-dilator} and \emph{$\gira$-dilator} in \cite{ecalle15}.
  On the other hand, Nao Komiyama kindly suggested usage of the \emph{$\mmu$-dilator} introduced in loc.~cit., which is the method also used in \cite{schneps15} to show a special case of the bisymmetrality of $\fess$.
  However, the author has not made a success of a proof in that way.
\end{remark}

\begin{lemma}\label{lem:negelon}
  Let $k,l,h$ and $r$ be non-negative integers satisfying $r\ge 2$, $h\ge 1$ and $k+l+h\le r-1$.
  Then we have
  \begin{align}
    F_{k,l,h}^{r}
    &\coloneqq\sum_{1\le j\le s\le r}\frac{s+1-j}{s(s+1)}\sum_{\substack{0\le c\le j-1\\ 0\le d\le s-j}}(-1)^{c+d}\binom{j-1}{c}\binom{s-j}{d}\binom{c}{k}\binom{d+1}{l}\binom{c+d+1}{h}\\
    &=0.
  \end{align}
\end{lemma}
\begin{proof}
  In the computation
  \begin{multline}
    \sum_{1\le j\le s\le r}\frac{s+1-j}{s(s+1)}\sum_{\substack{0\le c\le j-1\\ 0\le d\le s-j}}(-1)^{c+d}\binom{j-1}{c}\binom{s-j}{d}\binom{c}{k}\binom{d+1}{l}\binom{c+d+1}{h}\\
    =\sum_{\substack{c,d\ge 0\\ c+d\le r-1}}\sum_{s=c+d+1}^{r}\sum_{j=c+1}^{s-d}\frac{s+1-j}{s(s+1)}\cdot(-1)^{c+d}\binom{j-1}{c}\binom{s-j}{d}\binom{c}{k}\binom{d+1}{l}\binom{c+d+1}{h},
  \end{multline}
  since the sum about $j$ is simplified as
  \begin{align}
    \sum_{j=c+1}^{s-d}(s+1-j)\binom{j-1}{c}\binom{s-j}{d}
    &=(d+1)\sum_{j=c+1}^{s-d}\binom{j-1}{c}\binom{s-j+1}{d+1}\\
    &=(d+1)\binom{s+1}{c+d+2}\\
    &=(d+1)\frac{s(s+1)}{(c+d+1)(c+d+2)}\binom{s-1}{c+d}
  \end{align}
  which is a Vandermonde-type identity, the original sum becomes
  \[
    F_{k,l,h}^{r}
    =\sum_{c,d=0}^{r-1}\sum_{s=c+d+1}^{r}(-1)^{c+d}\frac{d+1}{(c+d+1)(c+d+2)}\cdot\binom{s-1}{c+d}\binom{c}{k}\binom{d+1}{l}\binom{c+d+1}{h}.
  \]
  Next, we see that the sum about $s$ becomes
  \[\sum_{s=c+d+1}^{r}\binom{s-1}{c+d}=\sum_{s=c+d+1}^{r}\left(\binom{s}{c+d+1}-\binom{s-1}{c+d+1}\right)=\binom{r}{c+d+1},\]
  and thus
  \begin{align}
    F_{k,l,h}^{r}
    &=\sum_{c,d=0}^{r-1}(-1)^{c+d}\frac{d+1}{(c+d+1)(c+d+2)}\cdot\binom{r}{c+d+1}\binom{c}{k}\binom{d+1}{l}\binom{c+d+1}{h}\\
    &=\sum_{n=0}^{r-1}(-1)^{n}\frac{1}{(n+1)(n+2)}\binom{r}{n+1}\binom{n+1}{h}\sum_{d=0}^{n}(d+1)\binom{d+1}{l}\binom{n-d}{k}.
  \end{align}
  Here, we can use the identity
  \[\sum_{d=0}^{n}(d+1)\binom{d+1}{l}\binom{n-d}{k}=l\binom{n+2}{k+l+1}+(l+1)\binom{n+2}{k+l+2},\]
  which is easily shown by calculating the generating function, so that
  \begin{align}
    F_{k,l,h}^{r}
    &=\sum_{n=0}^{r-1}(-1)^{n}\frac{1}{(n+1)(n+2)}\binom{r}{n+1}\left(l\binom{n+2}{k+l+1}+(l+1)\binom{n+2}{k+l+2}\right)\\
    &=\sum_{n=0}^{r-1}(-1)^{n}\binom{r}{n+1}\frac{1}{h}\binom{n}{h-1}\left(\frac{l}{k+l+1}\binom{n+1}{k+l}+\frac{l+1}{k+l+2}\binom{n+1}{k+l+1}\right).
  \end{align}
  Now let us use the standard result of the theory of finite differences: it states that, for any positive integer $N$ and an integer $1\le d<N$, the equality
  \[\sum_{n=1}^{N}(-1)^{n}\binom{N}{n}n^{d}=0\]
  holds.
  Since our sum is
  \[F_{k,l,h}^{r}=-\sum_{n=1}^{r}(-1)^{n}\binom{r}{n}\frac{1}{h}\binom{n-1}{h-1}\left(\frac{l}{k+l+1}\binom{n}{k+l}+\frac{l+1}{k+l+2}\binom{n}{k+l+1}\right)\]
  and the factor in the summand starting from $1/h$ is a polynomial of $n$ with the degree $\le k+l+h-1<r$ and without the constant term unless $k+l=0$.
  Even in such a case, the constant term vanishes due to the factor $l/(k+l+1)$.
  Thus we conclude that $F_{k,l,h}^{r}=0$ for all $k,l,h$ and $r$. 
\end{proof}

The following is a variant of \cite[Lemma 7.7]{kawamura25}.
\begin{proposition}\label{prop:dilator_symmetral}
  Let $S\in\MU$ and $D\in\LU$ be two bimoulds related as
  \[\der(S)=\preari(S,D).\]
  Then, $D$ is alternal if and only if $S$ is symmetral.
\end{proposition}
\begin{proof}
  First we assume that $S$ is symmetral.
  Then we have
  \[(\id+\ep\der)(S)=S+\ep\preari(S,D)=\gari(S,1+\ep D),\]
  as an equality holding in the $R[\ep]$-module of bimoulds.
  Hence we have
  \[\ep D=\logari(\gari(\invgari(S),(\id+\ep\der)(S))).\]
  Since $\GARI_{\as}$ forms a group (\cite[Proposition 3.17]{kawamura25}) and $\id+\ep\der$ preserves symmetrality, we see that $\ep D$ is equal to $\logari(\text{something symmetral})$ and thus alternal by \cite[Theorem A.7]{komiyama21}.
  Here, one has $(\id+\ep\der)(S)\in\GARI_{\as}$ as
  \begin{align}
  (\id+\ep\der)(S)(\ba\sh\bb)
  &=(1+\ep(\ell(\ba)+\ell(\bb)))S(\ba\sh\bb)\\
  &=(1+\ep(\ell(\ba)+\ell(\bb)))S(\ba)S(\bb)\\
  &=(1+\ep\ell(\ba))S(\ba)\cdot(1+\ep\ell(\bb))S(\bb)\\
  &=(\id+\ep\der)(S)(\ba)\cdot(\id+\ep\der)(S)(\bb).
  \end{align}
  Next we prove the symmetrality of $S$, assuming $D\in\ARI_{\al}$.
  Let us recall the first identity in \cite[p.~721]{fk23}: it states
  \begin{multline}\label{eq:fk}
    \begin{split}
    \arit(B)(A)(\ba\sh\bb)
    =\sum_{\substack{\ba=\ba_{1}\ba_{2}\ba_{3}\\ \ba_{2},\ba_{3}\neq\emp}}A(\ba_{1}\ful{\ba_{2}}\ba_{3}\sh\bb)B(\ba_{2}\flr{\ba_{3}})
    -\sum_{\substack{\ba=\ba_{1}\ba_{2}\ba_{3}\\ \ba_{1},\ba_{2}\neq\emp}}A(\ba_{1}\fur{\ba_{2}}\ba_{3}\sh\bb)B(\fll{\ba_{1}}\ba_{2})\\
    +\sum_{\substack{\bb=\bb_{1}\bb_{2}\bb_{3}\\ \bb_{2},\bb_{3}\neq\emp}}A(\ba\sh\bb_{1}\ful{\bb_{2}}\bb_{3})B(\bb_{2}\flr{\bb_{3}})
    -\sum_{\substack{\bb=\bb_{1}\bb_{2}\bb_{3}\\ \bb_{1},\bb_{2}\neq\emp}}A(\ba\sh\bb_{1}\fur{\bb_{2}}\bb_{3})B(\fll{\bb_{1}}\bb_{2})
    \end{split}
  \end{multline}
  for every $A\in\BIMU$ and $B\in\ARI_{\al}$ and non-empty $\ba,\bb\in\fV$.
  Note that, the original argument assumes $A\in\ARI_{\al}$, but in the proof of the above equation in \cite{fk23}, that condition is not used.
  Applying this identity for $A=S$ and $B=D$ in our present setting, we see that
  \begin{align}\label{eq:dilator_symmetral}
    \begin{split}
    &\der(S)(\ba\sh\bb)\\
    &=\arit(D)(S)(\ba\sh\bb)+\mmu(S,D)(\ba\sh\bb)\\
    &=\begin{multlined}[t]
      \sum_{\substack{\ba=\ba_{1}\ba_{2}\ba_{3}\\ \ba_{2},\ba_{3}\neq\emp}}S(\ba_{1}\ful{\ba_{2}}\ba_{3}\sh\bb)D(\ba_{2}\flr{\ba_{3}})
    -\sum_{\substack{\ba=\ba_{1}\ba_{2}\ba_{3}\\ \ba_{1},\ba_{2}\neq\emp}}S(\ba_{1}\fur{\ba_{2}}\ba_{3}\sh\bb)D(\fll{\ba_{1}}\ba_{2})\\
    +\sum_{\substack{\bb=\bb_{1}\bb_{2}\bb_{3}\\ \bb_{2},\bb_{3}\neq\emp}}S(\ba\sh\bb_{1}\ful{\bb_{2}}\bb_{3})D(\bb_{2}\flr{\bb_{3}})
    -\sum_{\substack{\bb=\bb_{1}\bb_{2}\bb_{3}\\ \bb_{1},\bb_{2}\neq\emp}}S(\ba\sh\bb_{1}\fur{\bb_{2}}\bb_{3})D(\fll{\bb_{1}}\bb_{2})
    +\mmu(S,D)(\ba\sh\bb).
    \end{multlined}
  \end{split}
  \end{align}
  We prove the symmetrality of $S$ by induction with the help of the above identity.
  The first symmetrality equation $S(a\sh b)=S(a)S(b)$ is checked by the fact that $D(a)=S(a)$ and
  \[S(ab)=\arit(D)(S)(ab)+\mmu(S,D)(ab)=S(a)D(b)=D(a)D(b),\]
  hold for letters $a$ and $b$.
  For the induction part, by applying the induction hypothesis for \eqref{eq:dilator_symmetral}, we have
  \begin{align}
    &\der(S)(\ba\sh\bb)\\
    &=\begin{multlined}[t]
      \sum_{\substack{\ba=\ba_{1}\ba_{2}\ba_{3}\\ \ba_{2},\ba_{3}\neq\emp}}S(\ba_{1}\ful{\ba_{2}}\ba_{3})S(\bb)D(\ba_{2}\flr{\ba_{3}})
    -\sum_{\substack{\ba=\ba_{1}\ba_{2}\ba_{3}\\ \ba_{1},\ba_{2}\neq\emp}}S(\ba_{1}\fur{\ba_{2}}\ba_{3})S(\bb)D(\fll{\ba_{1}}\ba_{2})\\
    +\sum_{\substack{\bb=\bb_{1}\bb_{2}\bb_{3}\\ \bb_{2},\bb_{3}\neq\emp}}S(\ba)S(\bb_{1}\ful{\bb_{2}}\bb_{3})D(\bb_{2}\flr{\bb_{3}})
    -\sum_{\substack{\bb=\bb_{1}\bb_{2}\bb_{3}\\ \bb_{1},\bb_{2}\neq\emp}}S(\ba)S(\bb_{1}\fur{\bb_{2}}\bb_{3})D(\fll{\bb_{1}}\bb_{2})
    +\mmu(S,D)(\ba\sh\bb)
    \end{multlined}\\
    &=\arit(D)(S)(\ba)S(\bb)+S(\ba)\arit(D)(S)(\bb)+\mmu(S,D)(\ba\sh\bb).
  \end{align}  
  On the other hand, using $D\in\ARI_{\al}$ we have
    \begin{align}
      \mmu(S,D)(\ba\sh\bb)
      &=\sum_{\substack{\ba=\ba_{1}\ba_{2}\\ \bb=\bb_{1}\bb_{2}}}S(\ba_{1}\sh\bb_{1})D(\ba_{2}\sh\bb_{2})\\
      &=\sum_{\substack{\ba=\ba_{1}\ba_{2}\\ \bb=\bb_{1}\bb_{2}\\ \ba_{2}=\emp\text{ or }\bb_{2}=\emp}}S(\ba_{1}\sh\bb_{1})D(\ba_{2}\sh\bb_{2})\\
      &=\sum_{\ba=\ba_{1}\ba_{2}}S(\ba_{1}\sh\bb)D(\ba_{2})+\sum_{\bb=\bb_{1}\bb_{2}}S(\ba\sh\bb_{1})D(\bb_{2})\\
      &=\sum_{\ba=\ba_{1}\ba_{2}}S(\ba_{1})S(\bb)D(\ba_{2})+\sum_{\bb=\bb_{1}\bb_{2}}S(\ba)S(\bb_{1})D(\bb_{2})\\
      &=\mmu(S,D)(\ba)S(\bb)+S(\ba)\mmu(S,D)(\bb),
    \end{align}
    where we used $D(\emp)=0$ and the induction hypothesis in the last equality.
    Combining these equations, we obtain
    \begin{align}
      (\ell(\ba)+\ell(\bb))S(\ba\sh\bb)
      &=\der(S)(\ba\sh\bb)\\
      &=\begin{multlined}[t]\arit(D)(S)(\ba)S(\bb)+S(\ba)\arit(D)(S)(\bb)+\sum_{\ba=\ba_{1}\ba_{2}}S(\ba_{1}\sh\bb)D(\ba_{2})\\+\mmu(S,D)(\ba)S(\bb)+S(\ba)\mmu(S,D)(\bb)\end{multlined}\\
      &=\preari(S,D)(\ba)S(\bb)+S(\ba)\preari(S,D)(\bb)\\
      &=\der(S)(\ba)S(\bb)+S(\ba)\der(S)(\bb)\\
      &=(\ell(\ba)+\ell(\bb))S(\ba)S(\bb).
      \end{align}
      This shows that $S$ is symmetral.
\end{proof}

\begin{lemma}\label{lem:mu_factor}
  Let $N$ be a positive integer and $S\in\MU$.
  Then we have
  \[\mmu^{N}(S)(\bw)=\sum_{i=0}^{\ell(\bw)}\binom{N}{i}\mmu^{i}(S-1)(\bw),\]
  where $\mmu^{n}(A)$ denotes the bimould $\mmu(\underbrace{A,\ldots,A}_{n})$.
\end{lemma}
\begin{proof}
  By the definition of $\mmu$, we can compute
  \begin{align}
    \mmu^{N}(S)(\bw)
    &=\sum_{\bw=\ba_{1}\cdots\ba_{N}}S(\ba_{1})\cdots S(\ba_{N})\\
    &=\sum_{i=0}^{\ell(\bw)}\sum_{\substack{\bw=\ba_{1}\cdots\ba_{N}\\ |\{j\mid \ba_{j}\neq\emp\}|=i}}S(\ba_{1})\cdots S(\ba_{N})\\
    &=\sum_{i=0}^{\ell(\bw)}S(\emp)^{N-i}\binom{N}{N-i}\sum_{\substack{\bw=\bb_{1}\cdots\bb_{i}\\ \bb_{1},\ldots,\bb_{i}\neq\emp}}S(\bb_{1})\cdots S(\bb_{i})\\
    &=\sum_{i=0}^{\ell(\bw)}\binom{N}{N-i}\mmu^{i}(S-1)(\bw).
  \end{align}
\end{proof}

\begin{lemma}\label{lem:shuffle}
  Let $A$ and $B$ be two bimoulds in $\MU_{\as}$.
  Then we have
  \begin{multline}
    \gaxit(A,B)(\fO)(\ba\sh\bb)
    =\sum_{\ba=\ba_{1}x\ba_{2}}A(\ba_{1}\flr{x})\fO(\ful{\ba_{1}}x\fur{\ba_{2}}\fur{\bb})B(\fll{x}\ba_{2})\mmu(A,B)(\fll{x}\bb)\\+\sum_{\bb=\bb_{1}y\bb_{2}}A(\bb_{1}\flr{y})\fO(\ful{\ba}\ful{\bb_{1}}y\fur{\bb_{2}})B(\fll{y}\bb_{2})\mmu(A,B)(\ba\flr{y}),
  \end{multline}
  for every $\ba,\bb\in\fV\setminus\{\emp\}$.
\end{lemma}
\begin{proof}
  The computation proceeds combinatorially:
  \begin{align}
    &\gaxit(A,B)(\fO)(\ba\sh\bb)\\
    &=\begin{multlined}[t]\sum_{\substack{\ba=\ba_{1}x\ba_{2}\\\bb=\bb_{1}\bb_{2}}}A((\ba_{1}\sh\bb_{1})\flr{x})\fO(\ful{\ba_{1}\bb_{1}}x\fur{\ba_{2}\bb_{2}})B(\fll{x}(\ba_{2}\sh\bb_{2}))\\
      +\sum_{\substack{\ba=\ba_{1}\ba_{2}\\\bb=\bb_{1}y\bb_{2}}}A((\ba_{1}\sh\bb_{1})\flr{y})\fO(\ful{\ba_{1}\bb_{1}}y\fur{\ba_{2}\bb_{2}})B(\fll{y}(\ba_{2}\sh\bb_{2}))
    \end{multlined}\\
    &=\begin{multlined}[t]\sum_{\substack{\ba=\ba_{1}x\ba_{2}\\\bb=\bb_{1}\bb_{2}}}A(\ba_{1}\flr{x}\sh\bb_{1}\flr{x})\fO(\ful{\ba_{1}\bb_{1}}x\fur{\ba_{2}\bb_{2}})B(\fll{x}\ba_{2}\sh\fll{x}\bb_{2})\\
      +\sum_{\substack{\ba=\ba_{1}\ba_{2}\\\bb=\bb_{1}y\bb_{2}}}A(\ba_{1}\flr{y}\sh\bb_{1}\flr{y})\fO(\ful{\ba_{1}\bb_{1}}y\fur{\ba_{2}\bb_{2}})B(\fll{y}\ba_{2}\sh\fll{y}\bb_{2})
    \end{multlined}\\
    &=\begin{multlined}[t]\sum_{\substack{\ba=\ba_{1}x\ba_{2}\\\bb=\bb_{1}\bb_{2}}}A(\ba_{1}\flr{x})A(\bb_{1}\flr{x})\fO(\ful{\ba_{1}\bb_{1}}x\fur{\ba_{2}\bb_{2}})B(\fll{x}\ba_{2})B(\fll{x}\bb_{2})\\
      +\sum_{\substack{\ba=\ba_{1}\ba_{2}\\\bb=\bb_{1}y\bb_{2}}}A(\ba_{1}\flr{y})A(\bb_{1}\flr{y})\fO(\ful{\ba_{1}\bb_{1}}y\fur{\ba_{2}\bb_{2}})B(\fll{y}\ba_{2})B(\fll{y}\bb_{2})
    \end{multlined}\\
    &=\begin{multlined}[t]\sum_{\ba=\ba_{1}x\ba_{2}}A(\ba_{1}\flr{x})\fO(\ful{\ba_{1}\bb_{1}}x\fur{\ba_{2}\bb_{2}})B(\fll{x}\ba_{2})\sum_{\bb=\bb_{1}\bb_{2}}A(\bb_{1}\flr{x})B(\fll{x}\bb_{2})\\
      +\sum_{\bb=\bb_{1}y\bb_{2}}A(\bb_{1}\flr{y})\fO(\ful{\ba_{1}\bb_{1}}y\fur{\ba_{2}\bb_{2}})B(\fll{y}\bb_{2})\sum_{\ba=\ba_{1}\ba_{2}}A(\ba_{1}\flr{y})B(\fll{y}\ba_{2})
    \end{multlined}\\
    &=\begin{multlined}[t]\sum_{\ba=\ba_{1}x\ba_{2}}A(\ba_{1}\flr{x})\fO(\ful{\ba_{1}}x\fur{\ba_{2}}\fur{\bb})B(\fll{x}\ba_{2})\mmu(A,B)(\fll{x}\bb)\\
      +\sum_{\bb=\bb_{1}y\bb_{2}}A(\bb_{1}\flr{y})\fO(\ful{\ba}\ful{\bb_{1}}y\fur{\bb_{2}})B(\fll{y}\bb_{2})\mmu(A,B)(\ba\flr{y}).
    \end{multlined}
  \end{align}
\end{proof}

\begin{theorem}[{\cite[(4.6)]{ecalle11}}]\label{thm:sro_dimorphy}
  The bimould
  \[\dTo(\re^{-1})=\sum_{r=1}^{\infty}\frac{1}{r(r+1)}\dro_{r}\]
  is $\fO$-alternal.
\end{theorem}
\begin{proof}
  Since $\ganit$ is homomorphism under the $\gani$-product (\cite[Proposition 3.10]{kawamura25}) and $\invgani(\foz)=(\pari\circ\anti)(\fos)$ (\cite[Proposition 5.3 (1)]{kawamura25}), we have
  \[
  \ganit(\foz)^{-1}(\dTo(\re^{-1}))(\bw)
  =\sum_{s=1}^{r}\sum_{\bw=b_{1}\bc_{1}\cdots b_{s}\bc_{s}}\dTo(\re^{-1})(b_{1}\fur{\bc_{1}}\cdots b_{s}\fur{\bc_{s}})\prod_{i=1}^{s}(\pari\circ\anti)(\fos)(\fll{b_{i}}\bc_{i})
  \]
  due to \cite[Lemma 3.9]{kawamura25}, for a variable sequence $\bw\in\fV$ with $\ell(\bw)=r\ge 1$. Here each bold symbol (resp.~italic letter) denotes a subsequence of $\bw$ (resp.~letter included in $\bw$).
  Using the explicit formula
  \[\dro_{r}(b'_{1},\ldots,b'_{s})=\sum_{j=1}^{s}(s+1-j)\foz((b'_{1}\cdots b'_{j-1})\flr{b'_{j}})\fO(\ful{b'_{1}\cdots b'_{j-1}}b'_{j}\fur{b'_{j+1}\cdots b'_{s}})\foz(\fll{b'_{j}}(b'_{j+1}\cdots b'_{s}))\]
  of $\dro$ (\cite[Lemma 6.8]{kawamura25}), we have
  \begin{align}
    &\ganit(\foz)^{-1}(\dTo(\re^{-1}))(\bw)\\
    &=\begin{multlined}[t]\sum_{s=1}^{r}\sum_{\bw=b_{1}\bc_{1}\cdots b_{s}\bc_{s}}\sum_{j=1}^{s}\frac{s+1-j}{s(s+1)}\foz\left((b_{1}\fur{\bc_{1}}\cdots b_{j-1}\fur{\bc_{j}})\flr{(b_{j}\fur{\bc_{j}})}\right)\\
      \cdot\fO\left(\ful{\left(\prod_{h=1}^{j-1}b_{h}\fur{\bc_{h}}\right)}(b_{j}\fur{\bc_{j}})\fur{\left(\prod_{h=j+1}^{s}b_{h}\fur{\bc_{h}}\right)}\right)\foz\left(\fll{(b_{j}\fur{\bc_{j}})}(b_{j+1}\fur{\bc_{j+1}}\cdots b_{s}\fur{\bc_{s}})\right)\prod_{i=1}^{s}(\pari\circ\anti)(\fos)(\fll{b_{i}}\bc_{i})
    \end{multlined}\\
    &=\begin{multlined}[t]\sum_{s=1}^{r}\sum_{\bw=b_{1}\bc_{1}\cdots b_{s}\bc_{s}}\sum_{j=1}^{s}\frac{s+1-j}{s(s+1)}\left(\prod_{h=1}^{j-1}\fO(b_{h}\fur{\bc_{h}}\flr{b_{j}})\right)\fO\left(\ful{\left(\prod_{h=1}^{j-1}b_{h}\bc_{h}\right)}(b_{j}\fur{\bc_{j}})\fur{\left(\prod_{h=j+1}^{s}b_{h}\bc_{h}\right)}\right)\\
      \cdot\left(\prod_{h=j+1}^{s}\fO(\fll{b_{j}}b_{h}\fur{\bc_{h}})\right)\prod_{i=1}^{s}(\pari\circ\anti)(\fos)(\fll{b_{i}}\bc_{i})
    \end{multlined}\\
    &=\begin{multlined}[t]\sum_{s=1}^{r}\sum_{\bw=b_{1}\bc_{1}\cdots b_{s}\bc_{s}}\sum_{j=1}^{s}\frac{s+1-j}{s(s+1)}\left(\prod_{h=1}^{j-1}\fO(b_{h}\fur{\bc_{h}}\flr{b_{j}})(\pari\circ\anti)(\fos)(\fll{b_{h}}\bc_{h})\right)\\
      \cdot\fO\left(\ful{\left(\prod_{h=1}^{j-1}b_{h}\bc_{h}\right)}(b_{j}\fur{\bc_{j}})\fur{\left(\prod_{h=j+1}^{s}b_{h}\bc_{h}\right)}\right)(\pari\circ\anti)(\fos)(\fll{b_{j}}\bc_{j})\left(\prod_{h=j+1}^{s}\fO(\fll{b_{j}}b_{h}\fur{\bc_{h}})(\pari\circ\anti)(\fos)(\fll{b_{h}}\bc_{h})\right).
    \end{multlined}
  \end{align}
  For the above sum, writing
  \[b_{h}\bc_{h}=\binom{u_{h,0},u_{h,1},\ldots,u_{h,c_{h}}}{v_{h,0},\ldots,v_{h,c_{h}}},\]
  we can apply the equality
  \begin{align}
    &\fO(b_{h}\fur{\bc_{h}}\flr{b_{j}})(\pari\circ\anti)(\fos)(\fll{b_{h}}\bc_{h})\\
    &=(-1)^{c_{h}}\fO\binom{u_{h,0}+\cdots+u_{h,c_{h}}}{v_{h,0}-v_{j,0}}\cdot\fos\binom{u_{h,c_{h}},\ldots,u_{h,1}}{v_{h,c_{h}}-v_{h,0},\ldots,v_{h,1}-v_{h,0}}\\
    &=(-1)^{c_{h}}\fO\binom{u_{h,0}+\cdots+u_{h,c_{h}}}{v_{h,0}-v_{j,0}}\cdot\fO\binom{u_{h,c_{h}}}{v_{h,c_{h}}-v_{h,c_{h}-1}}\cdots\fO\binom{u_{h,1}+\cdots+u_{h,c_{h}}}{v_{h,1}-v_{h,0}}\\
    &=(-1)^{c_{h}}\fos\binom{u_{h,c_{h}},\ldots,u_{h,1},u_{h,0}}{v_{h,c_{h}}-v_{j,0},\ldots,v_{h,1}-v_{j,0},v_{h,0}-v_{j,0}}\\
    &=-(\pari\circ\anti)(\fos)((b_{h}\bc_{h})\flr{b_{j}}),
  \end{align}
  so that the following is obtained:
  \begin{align}
    &\ganit(\foz)^{-1}(\dTo(\re^{-1}))(\bw)\\
    &=\begin{multlined}[t]\sum_{s=1}^{r}\sum_{\bw=b_{1}\bc_{1}\cdots b_{s}\bc_{s}}\sum_{j=1}^{s}(-1)^{s-1}\frac{s+1-j}{s(s+1)}\left(\prod_{h=1}^{j-1}(\pari\circ\anti)(\fos)((b_{h}\bc_{h})\flr{b_{j}})\right)\\
      \cdot\fO\left(\ful{\left(\prod_{h=1}^{j-1}b_{h}\bc_{h}\right)}(b_{j}\fur{\bc_{j}})\fur{\left(\prod_{h=j+1}^{s}b_{h}\bc_{h}\right)}\right)(\pari\circ\anti)(\fos)(\fll{b_{j}}\bc_{j})\left(\prod_{h=j+1}^{s}(\pari\circ\anti)(\fos)(\fll{b_{j}}(b_{h}\bc_{h}))\right)
    \end{multlined}\\
    &=\begin{multlined}[t]\sum_{s=1}^{r}\sum_{j=1}^{s}(-1)^{s-1}\frac{s+1-j}{s(s+1)}\sum_{\bw=\bx b\bc\by}\fO(\ful{\bx}(b\fur{\bc})\fur{\by})(\pari\circ\anti)(\fos)(\fll{b}\bc)\\
      \cdot\left(\sum_{\substack{\bx=\bc'_{1}\cdots\bc'_{j-1}\\ \bc'_{1},\ldots,\bc'_{j-1}\neq\emp}}\prod_{h=1}^{j-1}(\pari\circ\anti)(\fos)((\bc'_{h})\flr{b})\right)\left(\sum_{\substack{\by=\bc'_{j+1}\cdots\bc'_{s}\\ \bc'_{j+1},\ldots,\bc'_{s}\neq\emp}}\prod_{h=j+1}^{s}(\pari\circ\anti)(\fos)(\fll{b}(\bc'_{h}))\right)
    \end{multlined}\\
    &=\begin{multlined}[t]\sum_{s=1}^{r}\sum_{j=1}^{s}\frac{s+1-j}{s(s+1)}\sum_{\bw=\bx b\bc\by}\fO(\ful{\bx}(b\fur{\bc})\fur{\by})(\pari\circ\anti)(\fos)(\fll{b}\bc)\mmu^{j-1}(1-(\pari\circ\anti)(\fos))(\bx\flr{b})\\
      \cdot\mmu^{s-j}(1-(\pari\circ\anti)(\fos))(\fll{b}\by)
    \end{multlined}\\
    &=\begin{multlined}[t]\sum_{s=1}^{r}\sum_{j=1}^{s}\frac{s+1-j}{s(s+1)}\sum_{\bw=\bx b\by'}\fO(\ful{\bx}b\fur{\by'})\mmu^{j-1}(1-(\pari\circ\anti)(\fos))(\bx\flr{b})\\
      \cdot\mmu\left((\pari\circ\anti)(\fos)\mmu^{s-j}(1-(\pari\circ\anti)(\fos))\right)(\fll{b}\by').
    \end{multlined}
  \end{align}
  By the binomial expansion, we moreover get the expression
  \begin{align}
    &\ganit(\foz)^{-1}(\dTo(\re^{-1}))(\bw)\\
    &=\begin{multlined}[t]\sum_{1\le j\le s\le r}\frac{s+1-j}{s(s+1)}\sum_{\bw=\bx b\by}\fO(\ful{\bx}(b\fur{\bc})\fur{\by})\sum_{c=0}^{j-1}(-1)^{c}\binom{j-1}{c}\mmu^{c}((\pari\circ\anti)(\fos))(\bx\flr{b})\\
      \cdot\sum_{d=0}^{s-j}(-1)^{d}\binom{s-j}{d}\mmu^{d+1}((\pari\circ\anti)(\fos))(\fll{b}\by)
    \end{multlined}\\
    &=\begin{multlined}[t]\sum_{1\le j\le s\le r}\frac{s+1-j}{s(s+1)}\sum_{\substack{0\le c\le j-1\\ 0\le d\le s-j}}(-1)^{c+d}\binom{j-1}{c}\binom{s-j}{d}\\
      \cdot\gaxit(\mmu^{c}((\pari\circ\anti)(\fos)),\mmu^{d+1}((\pari\circ\anti)(\fos)))(\fO)(\bw).
    \end{multlined}
  \end{align}
  Hence we have, by Lemma \ref{lem:shuffle}, for every non-empty $\ba$ and $\bb$,
  \begin{align}
    &\ganit(\foz)^{-1}(\dTo(\re^{-1}))(\ba\sh\bb)\\
    &=\begin{multlined}[t]\sum_{1\le j\le s\le a+b}\frac{s+1-j}{s(s+1)}\sum_{\substack{0\le c\le j-1\\ 0\le d\le s-j}}(-1)^{c+d}\binom{j-1}{c}\binom{s-j}{d}\\
    \cdot\left(\sum_{\ba=\ba_{1}x\ba_{2}}\mmu^{c}((\pari\circ\anti)(\fos))(\ba_{1}\flr{x})\fO(\ful{\ba_{1}}x\fur{\ba_{2}}\fur{\bb})\right.\\
    \cdot\mmu^{d+1}((\pari\circ\anti)(\fos))(\fll{x}\ba_{2})\mmu^{c+d+1}((\pari\circ\anti)(\fos))(\fll{x}\bb)\\
    +\sum_{\bb=\bb_{1}y\bb_{2}}\mmu^{c}((\pari\circ\anti)(\fos))(\bb_{1}\flr{y})\fO(\ful{\ba}\ful{\bb_{1}}y\fur{\bb_{2}})\\
    \left.\phantom{\sum_{\bb_{1}}}\cdot\mmu^{d+1}((\pari\circ\anti)(\fos))(\fll{y}\bb_{2})\mmu^{c+d+1}((\pari\circ\anti)(\fos))(\ba\flr{y})\right)
  \end{multlined}
  \end{align}
  Applying Lemma \ref{lem:mu_factor} for each power of $(\pari\circ\anti)(\fos)$, one shows
  \begin{align}
    &\ganit(\foz)^{-1}(\dTo(\re^{-1}))(\ba\sh\bb)\\
    &=\begin{multlined}[t]\sum_{1\le j\le s\le a+b}\frac{s+1-j}{s(s+1)}\sum_{\substack{0\le c\le j-1\\ 0\le d\le s-j}}(-1)^{c+d}\binom{j-1}{c}\binom{s-j}{d}\\
    \cdot\left(\sum_{\ba=\ba_{1}x\ba_{2}}\sum_{k=0}^{\ell(\ba_{1})}\binom{c}{k}\mmu^{k}((\pari\circ\anti)(\fos)-1)(\ba_{1}\flr{x})\fO(\ful{\ba_{1}}x\fur{\ba_{2}}\fur{\bb})\right.\\
    \cdot\sum_{l=0}^{\ell(\ba_{2})}\binom{d+1}{l}\mmu^{l}((\pari\circ\anti)(\fos)-1)(\fll{x}\ba_{2})\sum_{h=0}^{\ell(\bb)}\binom{c+d+1}{h}\mmu^{h}((\pari\circ\anti)(\fos)-1)(\fll{x}\bb)\\
    +\sum_{\bb=\bb_{1}y\bb_{2}}\sum_{k=0}^{\ell(\bb_{1})}\binom{c}{k}\mmu^{k}((\pari\circ\anti)(\fos)-1)(\bb_{1}\flr{y})\fO(\ful{\ba}\ful{\bb_{1}}y\fur{\bb_{2}})\\
    \left.\cdot\sum_{l=0}^{\ell(\bb_{2})}\binom{d+1}{l}\mmu^{l}((\pari\circ\anti)(\fos)-1)(\fll{y}\bb_{2})\sum_{h=0}^{\ell(\ba)}\binom{c+d+1}{l}\mmu^{l}((\pari\circ\anti)(\fos)-1)(\ba\flr{y})\right)
  \end{multlined}\\
    &=\begin{multlined}[t]
      \sum_{\ba=\ba_{1}x\ba_{2}}\sum_{\substack{0\le k\le \ell(\ba_{1})\\ 0\le l\le \ell(\ba_{2})\\ 1\le h\le \ell(\bb)}}F_{k,l,h}^{a+b}\mmu^{k}((\pari\circ\anti)(\fos)-1)(\ba_{1}\flr{x})\fO(\ful{\ba_{1}}x\fur{\ba_{2}}\fur{\bb})\\
      \cdot\mmu^{l}((\pari\circ\anti)(\fos)-1)(\fll{x}\ba_{2})\mmu^{h}((\pari\circ\anti)(\fos)-1)(\fll{x}\bb)\\
      +\sum_{\bb=\bb_{1}y\bb_{2}}\sum_{\substack{0\le k\le \ell(\bb_{1})\\ 0\le l\le \ell(\bb_{2})\\ 1\le h\le \ell(\ba)}}F_{k,l,h}^{a+b}\mmu^{k}((\pari\circ\anti)(\fos)-1)(\bb_{1}\flr{y})\fO(\ful{\ba}\ful{\bb_{1}}y\fur{\bb_{2}})\\
      \cdot\mmu^{l}((\pari\circ\anti)(\fos)-1)(\fll{y}\bb_{2})\mmu^{h}((\pari\circ\anti)(\fos)-1)(\ba\flr{y}).
    \end{multlined}
  \end{align} 
  Note that, in the previous sums, the $h=0$ terms vanish as $\ba,\bb\neq\emp$.
  Therefore we have $\ganit(\foz)^{-1}(\dTo(\re^{-1}))(\ba\sh\bb)=0$ due to Lemma \ref{lem:negelon}. 
\end{proof}

\begin{theorem}[$=$ Theorem \ref{thm:bisymmetral}]\label{thm:bisymmetral2}
  The bimoulds $\fess$ and $\doss$ are symmetral.
\end{theorem}
\begin{proof}
  As we mentioned already, we only show the assertion for $\doss$.
  In the proof of \cite[Theorem 7.9]{kawamura25}, we have the equality
  \[(\der\circ\invgari)(\doss)=\preari(\invgari(\doss),\ganit(\foz)^{-1}(\dTo(\re^{-1}))).\]
  Hence Theorem \ref{thm:sro_dimorphy} and Proposition \ref{prop:dilator_symmetral} yields the symmetrality of $\doss$.
\end{proof}

\end{document}